\documentclass[a4paper,10pt]{amsart}

\usepackage[english]{babel}
\usepackage[utf8]{inputenc}
\usepackage[T1]{fontenc}
\usepackage{lmodern}

\usepackage{amssymb}
\usepackage{amsmath}
\usepackage{amsthm}
\usepackage{bbm}

\usepackage[shortlabels]{enumitem}
\usepackage{tikz}
\usepackage{marginnote}
\usepackage{mathtools}
\usepackage{orcidlink}

\usepackage{hyperref}

\newcommand{\bbE}{\mathbb{E}}

\newcommand{\bbN}{\mathbb{N}}

\newcommand{\bbR}{\mathbb{R}}

\newcommand{\calF}{\mathcal{F}}

\newcommand{\calL}{\mathcal{L}}

\newcommand{\N}{\bbN}
\newcommand{\R}{\bbR}

\DeclarePairedDelimiter{\norm}{\lVert}{\rVert}
\DeclarePairedDelimiter{\abs}{\lvert}{\rvert}
\DeclarePairedDelimiter{\dual}{\langle}{\rangle}
\DeclarePairedDelimiter{\set}{\lbrace}{\rbrace}

\DeclareMathOperator{\one}{\mathbbm{1}}

\DeclareMathOperator{\fix}{fix}

\newcommand{\ud}{\mathrm{d}}

\newcommand{\oto}{\xrightarrow{\,\mathrm{o}\,}}
\newcommand{\uoto}{\xrightarrow{\,\mathrm{uo}\,}}

\theoremstyle{definition}
\newtheorem{definition}{Definition}[section]
\newtheorem{remark}[definition]{Remark}

\newtheorem{general_assumption}[definition]{General Assumption}

\theoremstyle{plain}
\newtheorem{proposition}[definition]{Proposition}
\newtheorem{lemma}[definition]{Lemma}
\newtheorem{theorem}[definition]{Theorem}
\newtheorem{corollary}[definition]{Corollary}

\numberwithin{equation}{section}

\begin{document}

\title[Banach's principle in vector lattices]{Banach's principle in vector lattices}

\author{Alexander Dobrick \orcidlink{0000-0002-3308-3581}}
\address[A.~Dobrick]{Alexander Dobrick, Christian-Albrechts-Universität~zu~Kiel (Alumni), Arbeitsbereich~Analysis, 24118 Kiel, Germany}
\email{alexander.dobrick.math@web.de}

\date{\today}

\begin{abstract}
	We develop an abstract Banach principle in vector lattices and apply it to obtain lattice-theoretic versions of the individual and maximal ergodic theorems, without recourse to any measure representation. Considering a sequence of bounded operators with values in a Dedekind $\sigma$-complete vector lattice endowed with a locally solid topology satisfying the $\sigma$-Lebesgue property, we prove that the set of points at which the sequence is order convergent is a closed subspace and coincides with the whole space whenever convergence holds on a dense subset. Investigating positive, power-bounded, mean ergodic operators on order continuous Banach lattices, we construct a topology on the universal completion induced by a strictly positive order continuous functional and prove that the Cesàro means converge in order in the universal completion and, in particular, uo-converge in the original lattice. Moreover, we introduce the notion of a superinvariant pair and derive a lattice-theoretic Hopf inequality together with a weak type estimate via band projections, which yields an abstract maximal ergodic theorem. A spectral-theoretic version of the theorem follows from classical Perron--Frobenius theory. Finally, we specialise the abstract framework to the model space $L^0(\Omega)$ and revisit the classical Banach principle, the Hopf--Dunford--Schwartz theorem and Doob's martingale convergence theorem.
\end{abstract}

\maketitle

\section{Introduction} \label{section:introduction}
The pointwise convergence of ergodic averages is a central theme of analysis, with roots reaching from Boltzmann's statistical mechanics to Birkhoff's individual ergodic theorem \cite{Birkhoff1931}. While Birkhoff's theorem is formulated for measure-preserving transformations, the functional-analytic viewpoint pioneered by Yosida and Kakutani \cite{Yosida1939} reveals that its essential mechanism is a Banach-space principle: Banach \cite{Banach1926} observed already in 1926 that the set of points at which a sequence of operators converges pointwise is closed, provided the associated maximal operator is continuous at the origin. This abstract principle was elaborated by Hopf \cite{Hopf1954} and by Dunford and Schwartz \cite{Dunford1956}, whose maximal inequality for positive contractions on $L^1$ has since become the standard route to the ergodic theorem. A remarkably short proof of Hopf's inequality, which the proof of our own maximal ergodic theorem follows morally, was later given by Garsia \cite{Garsia1965}. This proof isolates the maximal inequality as the sole analytic input. Modern and comprehensive accounts of the interplay between Banach's principle and the maximal inequality are given in
\cite[Chapter~11]{Eisner2015} (see also \cite{Eisner2025}). Maximal inequalities as the sole analytic input needed for convergence theorems not confined to ergodic averages are quite common throughout the literature: Doob's martingale convergence theorem \cite{Doob1953} rests on the same Banach-space mechanism (see Section~\ref{section:applications-and-relations-to-the-classical-theory}), and so do many other theorems from classical Fourier and harmonic analysis, e.g., the Lebesgue differentiation theorem (see \cite{Hardy1930}) or almost everywhere convergence theorems for approximate identities, Walsh--Fourier series (see \cite{Billard1967}), spherical maximal averages (see \cite{Bourgain1986, Stein1976}), Bochner--Riesz means (see \cite{Carbery1988}), Calderón--Zygmund singular integral operators (see \cite{Stein1970a}).

The classical theory, however, is tied to a measure-theoretic representation: maximal functions are defined via level sets, and Banach's closed-set argument is carried out inside a concrete space of measurable functions. From the perspective of vector lattices this dependence is unsatisfactory, and the wish to remove it is not new: as early as 1948, Nakano attempted an ergodic theorem for semi-ordered linear spaces \cite{Nakano1948}, introducing in the process a notion of unbounded convergence that was taken up by DeMarr \cite{DeMarr1964} and put on a systematic footing by Wickstead \cite{Wickstead1977} under the name of \emph{unbounded order convergence}, or \emph{uo-convergence} for short. Over the past decade uo-convergence has developed into a mature theory in its own right, with contributions by various authors (see \cite{Azouzi2019, DengOBrienTroitsky2017, Gao2014, GaoLeungXanthos2018, Gao2017, GaoXanthos2014, Kandic2017, Kandic2018, Li2018, Taylor2018, Taylor2019, VanDerWalt2018}). By now, it seems to be well understood, and central to the present article, that uo-convergence is exactly the representation-free substitute for almost everywhere convergence that Nakano was reaching for: on the model space $L^0(\Omega)$, uo-convergence of a net recovers almost everywhere convergence verbatim, with no reference to a dominating element or to order-boundedness.

A second lattice-theoretic ingredient is furnished by the theory of \emph{locally solid topologies} on vector lattices, developed systematically by Aliprantis and Burkinshaw \cite{Aliprantis2003} as well as many others. Of particular importance to us is the \emph{universal completion} $E^u$ of an Archimedean vector lattice $E$ (see \cite[Chapter~7]{Aliprantis2003}), which is Dedekind and laterally complete and provides a natural ambient space in which order-boundedness questions can be settled even when $E$ itself is not complete enough; in particular, the supremum of a family of Cesàro means may fail to exist in $E$ but always exists in $E^u$. The universal completion also furnishes the concrete representation underlying our approach: in \cite{Cerreia2022} the authors show that a universally complete vector lattice with a weak order unit carrying a strictly positive order continuous functional is lattice isomorphic to $L^0(\Omega)$ for some probability space $\Omega$. Such a functional induces a gauge metric, or actually even a gauge seminorm, on $E^u$ that metrises order convergence and plays the role that the measure itself plays in the classical theory. A noncommutative precursor to this circle of ideas is due to Goldstein and Litvinov \cite{Goldstein2000}, who proved a Banach principle in the space of $\tau$-measurable operators affiliated with a semifinite von Neumann algebra, endowed with the measure topology. It can be seen that the structural ingredients of their proof align closely with the lattice-theoretic axiomatics adopted below.

A structurally related attempt at removing measure theory from ergodic averaging is due to Kuo, Labuschagne and Watson \cite{Kuo2007}, who, building on the Riesz-space martingale calculus of \cite{Kuo2004}, obtained order-theoretic analogues of the Birkhoff, Hopf and Wiener theorems directly on Riesz spaces with a weak order unit by means of band projections. Their framework, however, presupposes a conditional expectation operator which acts as the lattice-theoretic counterpart of a sub-$\sigma$-algebra. The ensuing line of work has remained tied to this conditional setting, extending it to Poincar\'e recurrence, the Kac formula and Kakutani--Rokhlin decompositions (see \cite{Azouzi2023, Azouzi2024, Azouzi2025}). In contrast, our approach does not make use of conditional expectation operators.

The present article unifies these two strands. We replace the measure-theoretic maximal inequality by an order-theoretic boundedness condition in the universal completion, we replace level sets by band projections, and we replace the integral by the gauge metric. The resulting Banach principle, individual ergodic theorem, and maximal ergodic theorem are order-theoretic theorems in the strictest sense as they do not rely on any representation theorems. Nonetheless, classical results are recovered at the very end, as an application of the abstract theory to the model space $L^0(\Omega)$ in Section~\ref{section:applications-and-relations-to-the-classical-theory}.

As a first illustration of the abstract framework, we state the Banach principle in its lattice-theoretic form. Under a standing general assumption on a sequence of operators $(T_n)_{n \in \N}$ on a suitable vector lattice $E$, the convergence set
\begin{align*}
    C \coloneqq \set{x \in E : (T_n x)_{n \in \N} \text{ is order convergent in } E}
\end{align*}
is a closed subspace, and $C = E$ whenever the convergence can be established on a dense subset (see Theorem~\ref{theorem:banach-principle}). The second, and in our view the strongest, result of the article is the abstract individual ergodic theorem. Let $T$ be a positive, power-bounded, mean ergodic operator on an order continuous Banach lattice $E$ with a quasi-interior super fixed point $e \in E$. Then the Cesàro means $A_n x$ uo-converge to the mean ergodic projection $Px$ in $E$
for each $x \in E$ (see Theorem~\ref{theorem:individual-ergodic}).

\subsection*{Contributions of this article}
The present article establishes three main results, all formulated purely in the language of vector
lattices and without recourse to measure-theoretic arguments or representation theorems.

First, we prove an abstract Banach principle in vector lattices (see
Theorem~\ref{theorem:banach-principle}). Under a standing general assumption on a sequence of linear
operators with values in a Dedekind $\sigma$-complete vector lattice endowed with a locally solid
topology, we show that the set of points at which the sequence is order convergent is a closed subspace
and coincides with the whole space as soon as convergence holds on a dense subset. The proof rests on a
lattice-theoretic continuity result for the associated maximal operator, which we obtain via Baire's
category theorem.

Second, we prove the abstract individual ergodic theorem announced above (see
Theorem~\ref{theorem:individual-ergodic}). For a positive, power-bounded, mean ergodic operator on an
order continuous Banach lattice with a weak unit (equivalently, a quasi-interior point), we show that the
Cesàro means converge in order in the universal completion and, in particular, uo-converge in the original
lattice. The argument combines the abstract Banach principle with the gauge metric on the universal
completion induced by a strictly positive order continuous functional.

Third, we establish an abstract maximal ergodic theorem (see Theorem~\ref{theorem:maximal-ergodic}).
Introducing the notion of a superinvariant pair, we derive a lattice-theoretic version of Hopf's
inequality -- our proof of which is modelled on Garsia's argument \cite{Garsia1965} -- and pass to a
weak-type estimate within the universal completion. This yields the order boundedness of the Cesàro means
required by the individual ergodic theorem and, in particular, a spectral-theoretic version via classical
Perron--Frobenius theory.

\subsection*{Organization of the article}
The paper is organised in four sections. In Section~\ref{section:an-abstract-banach-principle} we develop the abstract Banach principle in vector lattices, proving that the convergence set of an operator sequence is a closed subspace and coincides with the whole space whenever convergence holds on a dense subset. Section~\ref{section:an-abstract-individual-ergodic-theorem} is devoted to the abstract individual ergodic theorem: after constructing a gauge metric on the universal completion, we show that the Cesàro means of a positive, power-bounded, mean ergodic operator converge in order. In Section~\ref{section:an-abstract-maximal-ergodic-theorem} we establish the maximal hypothesis by lattice-theoretic means, deriving an abstract Hopf inequality and a weak type estimate, and deduce an individual ergodic theorem under a superinvariant pair hypothesis. Finally, in Section~\ref{section:applications-and-relations-to-the-classical-theory} we specialise the abstract framework to the model space $L^0(\Omega)$ and recover the classical Banach principle, the Hopf--Dunford--Schwartz theorem, and Doob's martingale convergence theorem.

\subsection*{Basic notions and notation}

We collect some of the lattice-theoretic terminology used throughout this article. For a comprehensive treatment of vector and Banach lattices we refer to the classical monographs \cite{Aliprantis1999, Aliprantis2003, Aliprantis2006, Batkai2017, Luxemburg1971, MeyerNieberg1991, Schaefer1974, Vulikh1967, Zaanen1983, Zaanen2012}.

Let $E$ be a vector lattice. A vector lattice $E$ is called \emph{Archimedean} if $n x \le y$ for all $n \in \N$ implies $x \le 0$, and \emph{Dedekind complete} if every non-empty subset that is bounded above has a supremum. Two elements $x, y \in E$ are called \emph{disjoint}, written $x \perp y$, if $\inf \set{\abs{x}, \abs{y}} = 0$. A set $A \subseteq E$ is called \emph{order bounded} if it is contained in an interval of the form $[a, b] \coloneqq \set{x \in E : a \le x \le b}$ for some $a, b \in E$. An element $u \in E_+$ is called a \emph{weak unit} if $x \wedge u = 0$ implies $x = 0$ for each $x \in E_+$.

A subset $A \subseteq E$ is called \emph{solid} if $\abs{y} \le \abs{x}$ and $x \in A$ imply $y \in A$. A linear subspace $I \subseteq E$ that is simultaneously solid is called an \emph{ideal}; equivalently, $I$ is a subspace such that $x \in I$ and $\abs{y} \le \abs{x}$ imply $y \in I$. An ideal is called a \emph{band} if it is closed under suprema of arbitrary subsets. For $e \in E_+$, the \emph{principal ideal} generated by $e$ is $E_e \coloneqq \set{x \in E : \abs{x} \le \lambda e \text{ for some } \lambda > 0}$. For $A \subseteq E_+$, the \emph{band generated by $A$} is the smallest band containing $A$, denoted by $A^{\perp\perp}$. If $E$ is Dedekind complete, each band $B \subseteq E$ is a projection band, i.e., $E = B \oplus B^\perp$, and the associated \emph{band projection} $P_B \colon E \to E$ is the positive contractive projection onto $B$ with kernel $B^\perp$. In this case, one has $0 \leq P_B \leq I$, where $I$ denotes the identity on $E$. A sublattice $F \subseteq E$ is called \emph{order dense} if for each $0 < x \in E$ there exists $y \in F$ with $0 < y \le x$. A \emph{lattice isomorphism} is a bijective linear map between vector lattices that preserves the lattice operations in both directions.

Let $E$ be a Banach lattice. The norm is called \emph{order continuous} if $x_\alpha \downarrow 0$ implies $\norm{x_\alpha} \to 0$ for each decreasing net. An element $e \in E_+$ is called a \emph{quasi-interior point} if the principal ideal $E_e$ is norm dense in $E$; in a Banach lattice with order continuous norm, each weak unit is a quasi-interior point. A functional $\varphi \in E'$ is called \emph{strictly positive} if $\varphi(x) > 0$ for each $0 < x \in E_+$, and \emph{order continuous} (respectively \emph{$\sigma$-order continuous}) if $x_\alpha \downarrow 0$ (respectively $x_n \downarrow 0$) implies $\varphi(x_\alpha) \to 0$. The \emph{null ideal} of $\varphi \in E'_+$ is $N_\varphi \coloneqq \set{x \in E : \dual{\abs{x}, \varphi} = 0}$. An \emph{AL-space} is a Banach lattice whose norm is additive on $E_+$, and an \emph{AM-space} is a Banach lattice such that $\norm{x \vee y} = \max(\set{\norm{x}, \norm{y}})$ for all $x, y \in E_+$. We write $x_\alpha \downarrow 0$ for a net $(x_\alpha)$ in $E_+$ that is decreasing and satisfies $\inf_\alpha x_\alpha = 0$. Similarly, $x_\alpha \uparrow x$ denotes an increasing net with supremum $x$.

Let $T$ be a bounded operator on $E$. The \emph{fixed space} of $T$ is $\fix T \coloneqq \set{x \in E : Tx = x}$. A power-bounded operator $T$ on a Banach space is called \emph{mean ergodic} if the Cesàro means $A_n \coloneqq \frac{1}{n} \sum_{k=0}^{n-1} T^k$ converge strongly. In this case, the limit is a projection $P$, called the \emph{mean ergodic projection}. An element $e \in E_+$ is called a \emph{super fixed point} of $T$ if $Te \leq e$.

\section{An abstract Banach principle} \label{section:an-abstract-banach-principle}

Throughout this section, all vector lattices are assumed to be real and Ar\-chi\-me\-de\-an. Recall that a vector lattice $E$ is called \emph{Dedekind $\sigma$-complete} if every countable, non-empty subset of $E$ which is bounded above has a supremum.

\subsection*{Order convergence and locally solid topologies} \label{subsection:order-convergence-and-locally-solid-topologies}

Let $E$ be a vector lattice. A sequence $(x_n)_{n \in \N}$ in $E$ is said to be \emph{order convergent} to some $x \in E$ if there exists a sequence $(y_n)_{n \in \N}$ in $E$ such that $y_n \downarrow 0$ and $\abs{x_n - x} \leq y_n$ for each $n \in \N$. In this case, we write $x_n \oto x$.

In what follows, we shall make use of the following elementary lemma, which characterises order convergence of order bounded sequences by an oscillation criterion. It seems to be folklore but we were not able to find an exact reference.

\begin{lemma} \label{lemma:order-cauchy}
	Let $E$ be a Dedekind $\sigma$-complete vector lattice and let $(x_n)_{n \in \N}$ be an order bounded sequence in $E$. For each $N \in \N$ the element
	\begin{align*}
		s_N \coloneqq \sup_{n, m \geq N} \abs{x_n - x_m}
	\end{align*}
	exists, the sequence $(s_N)_{N \in \N}$ is decreasing and the following assertions are equivalent:
	\begin{enumerate}[(i)]
		\item The sequence $(x_n)_{n \in \N}$ is order convergent.
		\item One has $\inf_{N \in \N} s_N = 0$.
	\end{enumerate}
In this case, the limit $x \coloneqq \inf_{N \in \N} \sup_{n \geq N} x_n$ satisfies $\abs{x_n - x} \leq s_N$ for all $n \geq N$.
\end{lemma}

\begin{proof}
	Since $(x_n)_{n \in \N}$ is order bounded and $E$ is Dedekind $\sigma$-complete, the elements $s_N$ and $b_N \coloneqq \sup_{n \geq N} x_n$ exist for each $N \in \N$. Moreover, note that both sequences $(s_N)_{N \in \N}$ and $(b_N)_{N \in \N}$ are decreasing.

	(i) $\Rightarrow$ (ii): Suppose that $(x_n)_{n \in \N}$ is order convergent to some $x \in E$. Then there exists $(y_n)_{n \in \N}$ in $E$ such that $y_n \downarrow 0$ and $\abs{x_n - x} \leq y_n$ for all $n \in \N$. Then $\abs{x_n - x_m} \leq y_n + y_m \leq 2 y_N$ for all $n, m \geq N$. Therefore, $0 \leq s_N \leq 2 y_N \downarrow 0$.

	(ii) $\Rightarrow$ (i): Set $x \coloneqq \inf_{N \in \N} b_N$. Fix $N \in \N$ and $n \geq N$. On the one hand, $x_m \leq x_n + s_N$ for all $m \geq N$, so $x \leq b_N \leq x_n + s_N$. On the other hand, $x_m \geq x_n - s_N$ for all $m \geq N$ yields $b_{N'} \geq x_n - s_N$ for every $N' \geq N$ and thus $x = \inf_{N' \geq N} b_{N'} \geq x_n - s_N$. Consequently, $\abs{x_m - x} \leq s_N$ for all $m \geq N$. In particular, one has $\abs{x_n - x} \leq s_n$ for each $n \in \N$. By assumption, one has $s_n \downarrow 0$ and thus $x_n \oto x$.
\end{proof}

Next, we recall some topological notions that are central for this paper. A linear topology $\tau$ on a vector lattice $E$ is called \emph{locally solid} if it admits a neighbourhood base of zero consisting of solid sets (see \cite[Definition~2.16]{Aliprantis2003}). In this case, the lattice operations on $E$ are uniformly $\tau$-continuous by \cite[Theorem~2.17]{Aliprantis2003}. Furthermore, $\tau$ is said to have the \emph{$\sigma$-Lebesgue property} if $x_n \downarrow 0$ implies $x_n \xrightarrow{\tau} 0$ (see \cite[Definition~3.1]{Aliprantis2003}). Note that, in this case, each order convergent sequence is $\tau$-convergent to the same limit.

The following definition can be found in \cite[Definition~2.27]{Aliprantis2003}: A \emph{Riesz pseudonorm} on $E$ is a map $\rho \colon E \to [0, \infty)$ such that the following assertions hold:
\begin{enumerate}
    \item[(a)] $\rho(x + y) \leq \rho(x) + \rho(y)$ for all $x, y \in E$.
    \item[(b)] $\rho(\lambda x) \to 0$ as $\lambda \to 0$ for each $x \in E$.
    \item[(c)] If $\abs{x} \leq \abs{y}$ for $x,y \in E$, then $\rho(x) \leq \rho(y)$.
\end{enumerate}

It is a well-known fact that locally solid topologies are exactly those generated by families of Riesz pseudonorms (see \cite[Theorem~2.28]{Aliprantis2003}). Moreover, the following lemma holds.

\begin{lemma} \label{lemma:f-norm}
	Let $\tau$ be a metrisable, locally solid linear Hausdorff topology on a vector lattice $E$. Then there exists a Riesz pseudonorm $\rho$ on $E$ which generates $\tau$.
\end{lemma}

\begin{proof}
	Since $\tau$ is locally solid, it is generated by a family of Riesz pseudonorms (see \cite[Theorem~2.28]{Aliprantis2003}). As $\tau$ is metrisable, it possesses a countable neighbourhood base of zero, so a countable subfamily $(\rho_n)_{n \in \N}$ already generates $\tau$. Clearly, the standard construction
    \begin{align*}
        \rho \colon E \to [0, \infty), \quad \rho(x) = \sum_{n \in \N} 2^{-n} \min \set{\rho_n(x), 1}
    \end{align*}
    then yields a Riesz pseudonorm which generates $\tau$.
\end{proof}

In the setting of the above lemma, observe that the Riesz pseudonorm $\rho$ which generates $\tau$ is itself $\tau$-continuous since one has
\begin{align*}
    \abs{\rho(x) - \rho(y)} \leq \rho(x - y) \qquad \text{for all } x, y \in E.
\end{align*}
by the inverse triangle inequality.

\subsection*{A Banach principle in vector lattices} \label{subsection:a-banach-principle-in-vector-lattices}

We are now in position to present an abstract Banach principle in vector lattices. As mentioned in the introduction, the main ingredient in many concrete Banach principles is order.
Throughout this subsection, we will work in an abstract setting. In particular, we will work under the following general assumptions.

\begin{general_assumption} \label{general-assumption:standing} Let $X$ be a Banach space and let $E$ be a vector lattice. Further suppose that:
	\begin{enumerate}
		\item[(a)] $E$ is Dedekind $\sigma$-complete.
        \item[(b)] There exists a metrisable, Hausdorff, locally solid linear topology $\tau$ on $E$ that has the $\sigma$-Lebesgue property.
		\item[(c)] There are linear operators $T_n \colon X \to E$, $n \in \N$, that are continuous with respect to the norm topology on $X$ and the topology $\tau$ on $E$.
		\item[(d)] For each $x \in X$ the sequence $(T_n x)_{n \in \N}$ is order bounded in $E$.
	\end{enumerate}
\end{general_assumption}

Under the general assumptions above, we are able to define the \emph{maximal operator}
\begin{align*}
    T^\star \colon X \to E_+, \qquad T^\star x \coloneqq \sup_{n \in \N} \abs{T_n x}.
\end{align*}
Using General~Assumption~\ref{general-assumption:standing}(a)~and~(d), it is easy to see that this operator is well-defined. Note that $T^\star$ is subadditive and satisfies $T^\star(\lambda x) = \abs{\lambda} T^\star x$ for all $\lambda \in \R$ and $x \in X$. Using Baire's category theorem, one can show that the maximal operator is continuous at the origin.

\begin{lemma} \label{lemma:maximal-operator-continuity}
    Let $X$ be a Banach space and let $E$ be a vector lattice. Suppose that General Assumption~\ref{general-assumption:standing} holds. Then the maximal operator $T^\star$ is continuous at $0$.
\end{lemma}

\begin{proof}
    By General Assumption~\ref{general-assumption:standing}~(b) and Lemma~\ref{lemma:f-norm}, the topology $\tau$ on $E$ is generated by a Riesz pseudonorm $\rho$. For each $N \in \N$ we define the mapping
    \begin{align*}
        a_N \colon X \to E_+, \quad a_N(x) \coloneqq \max_{n \leq N} \abs{T_n x}.
    \end{align*}
    Since the lattice operations are $\tau$-continuous and each $T_n$ is continuous with respect to $\tau$, it follows that each $a_N$ is continuous with respect to $\tau$. Moreover, one clearly has $a_N(x) \uparrow T^\star x$ as $N \to \infty$ for each $x \in X$. Thus, the $\sigma$-Lebesgue property (cf.\ General Assumption~\ref{general-assumption:standing}~(b)) yields $a_N(x) \xrightarrow{\tau} T^\star x$ and hence
	\begin{align} \label{eq:sup-formula}
		\rho(\lambda a_N(x)) \uparrow \rho(\lambda T^\star x)
	\end{align}
	as $N \to \infty$ for each $\lambda > 0$ and $x \in X$ by the monotonicity of $\rho$.

    Fix $\varepsilon > 0$. Consider the closed sets
	\begin{align*}
		A_k \coloneqq \set[\big]{x \in X : \rho\big(\tfrac{1}{k} a_N(x)\big) \leq \varepsilon \text{ for all } N \in \N}, \ k \in \N.
	\end{align*}
	By \eqref{eq:sup-formula}, one has $x \in A_k$ if and only if $\rho(\tfrac{1}{k} T^\star x) \leq \varepsilon$. Since $\rho(\tfrac{1}{k} T^\star x) \to 0$ as $k \to \infty$ for each fixed $x \in X$, this shows that $X = \bigcup_{k \in \N} A_k$. Hence, by the Baire category theorem there exist $k_0 \in \N$, $x_0 \in X$ and $\delta > 0$ such that $\rho(\tfrac{1}{k_0} T^\star z) \leq \varepsilon$ whenever $\norm{z - x_0} \leq \delta$. Now, let $y \in X$ with $\norm{y} \leq \delta$. By subadditivity, one has $T^\star y \leq T^\star(x_0 + y) + T^\star x_0$ and therefore
	\begin{align*}
		\tfrac{1}{k_0} T^\star y \leq \tfrac{1}{k_0} T^\star (x_0 + y) + \tfrac{1}{k_0} T^\star x_0.
	\end{align*}
	Consequently, the monotonicity and the subadditivity of $\rho$ imply
	\begin{align*}
		\rho(\tfrac{1}{k_0} T^\star y) \leq \rho(\tfrac{1}{k_0} T^\star (x_0 + y)) + \rho(\tfrac{1}{k_0} T^\star x_0) \leq 2 \varepsilon.
	\end{align*}
	Finally, let $x \in X$ with $\norm{x} \leq \tfrac{\delta}{k_0}$. Writing $x = \tfrac{1}{k_0} y$ with $\norm{y} \leq \delta$ and using the positive homogeneity of $T^\star$, we obtain $\rho(T^\star x) = \rho(\tfrac{1}{k_0} T^\star y) \leq 2 \varepsilon$. Thus, $T^\star$ is continuous at $0$.
\end{proof}

Next, we state the main theorem of this section, an abstract lattice-theoretic version of Banach's principle.

\begin{theorem}[Banach's principle] \label{theorem:banach-principle}
	Let $X$ be a Banach space and let $E$ be a vector lattice. Suppose that General Assumption~\ref{general-assumption:standing} holds. Then the following two assertions hold:
	\begin{enumerate}[\upshape (i)]
		\item The \emph{convergence set}
		\begin{align*}
			C \coloneqq \set{x \in X : (T_n x)_{n \in \N} \text{ is order convergent in } E}
		\end{align*}
		is a closed subspace of $X$.

		\item If there exists a dense subset $D \subseteq X$ such that $(T_n x)_{n \in \N}$ is order convergent for each $x \in D$, then $C = X$.
	\end{enumerate}
\end{theorem}

\begin{proof}
	(i): Clearly, $C$ is a subspace of $X$. To show its closedness, let $(x_k)_{k \in \N}$ be a sequence in $C$ converging to some $x \in X$. As $(T_n x)_{n \in \N}$ is order bounded, Lemma~\ref{lemma:order-cauchy} is applicable and it suffices to show that the oscillation
	\begin{align*}
		y \coloneqq \inf_{K \in \N} \sup_{n, m \geq K} \abs{T_n x - T_m x}
	\end{align*}
	vanishes. To this end, fix $k \in \N$. For every $K \in \N$ and all $n, m \geq K$ one has
	\begin{align*}
		\abs{T_n x - T_m x} &\leq \abs{T_n(x - x_k)} + \abs{T_m(x - x_k)} + \abs{T_n x_k - T_m x_k} \\
        &\leq 2 \, T^\star(x - x_k) + \sup_{n, m \geq K} \abs{T_n x_k - T_m x_k}.
	\end{align*}
    Taking the supremum over $n, m \geq K$, we obtain
    \begin{align*}
        \sup_{n, m \geq K} \abs{T_n x - T_m x} \leq 2 \, T^\star(x - x_k) + \sup_{n, m \geq K} \abs{T_n x_k - T_m x_k}.
    \end{align*}

	Taking the infimum over $K \in \N$, we conclude from Lemma~\ref{lemma:order-cauchy}, applied to the order convergent sequence $(T_n x_k)_{n \in \N}$, that
	\begin{align*}
		0 \leq y \leq 2 \, T^\star (x - x_k)
	\end{align*}
    for all $k \in \N$. By General Assumption~\ref{general-assumption:standing}~(b) and Lemma~\ref{lemma:f-norm}, the topology $\tau$ on $E$ is generated by a Riesz pseudonorm $\rho$. Since $x - x_k \to 0$ in $X$, Lemma~\ref{lemma:maximal-operator-continuity} and the monotonicity of $\rho$ imply $\rho(y) \leq 2 \rho(T^\star(x - x_k)) \to 0$ as $k \to \infty$. Thus, $\rho(y) = 0$ and since $\tau$ is Hausdorff, we obtain $y = 0$.

	(ii): By assertion~(i), $C$ is a closed subspace of $X$ containing the dense set $D$, whence $C = X$.
\end{proof}

Theorem~\ref{theorem:banach-principle} yields a reduction scheme which underlies the applications in Section~\ref{section:applications-and-relations-to-the-classical-theory}: to establish order convergence of the operators on the whole space, it suffices to verify the maximal hypothesis and to exhibit a dense class on which each orbit order converges.

\begin{remark} \label{remark:uo}
	Under General Assumption~\ref{general-assumption:standing}(d), every orbit $(T_n x)_{n \in \N}$ is order bounded, so its order convergence is equivalent to its uo-convergence. Consequently, the convergence set of Theorem~\ref{theorem:banach-principle} can equivalently be written as
	\begin{align*}
		C = \set{x \in X : (T_n x)_{n \in \N} \text{ is uo-convergent in } E}.
	\end{align*}
	In the model case discussed in Section~\ref{section:applications-and-relations-to-the-classical-theory} this is precisely the mechanism of the classical principle: the maximal hypothesis upgrades plain almost everywhere convergence, which corresponds to uo-convergence, to dominated almost everywhere convergence, which corresponds to order convergence.
\end{remark}

In our abstract framework, we can revisit parts of \cite[Theorem~7.52]{Eisner2025}.

\begin{corollary} \label{corollary:banach-principle-equivalent-forms}
    Let $X$ be a Banach space and let $E$ be a vector lattice. Suppose that General Assumption~\ref{general-assumption:standing}\upshape{(a)--(c)} hold. Consider the following assertions:
    \begin{enumerate}[\upshape (i)]
        \item $(T_n x)_{n \in \N}$ is order convergent in $E$ for every $x \in X$.
        \item $(T_n x)_{n \in \N}$ is order bounded in $E$ for every $x \in X$, i.e., General Assumption~\ref{general-assumption:standing}\upshape{(d)} holds.
        \item The convergence set
        \begin{align*}
            C \coloneqq \set{x \in X : (T_n x)_{n \in \N} \text{ is order convergent in } E}
        \end{align*}
        is a closed subspace of $X$.
    \end{enumerate}
    Then \upshape{(i)} $\Rightarrow$ \upshape{(ii)} $\Rightarrow$ \upshape{(iii)}. Moreover, one has \upshape{(iii)} $\Rightarrow$ \upshape{(i)} if there exists a dense subset $D \subseteq X$ such that $(T_nx)_{n\in\N}$ is order convergent in $E$ for every $x \in D$.
\end{corollary}

\begin{proof}
    (i) $\Rightarrow$ (ii): Fix $x \in X$ and suppose $T_n x \oto z$ for some $z \in E$. Then there exists $(y_n)_{n \in \N}$ in $E$ such that $y_n \downarrow 0$ and $\abs{T_n x - z}\leq y_n$ for all $n \in \N$. In particular, $\abs{T_n x} \leq \abs z + y_n \leq \abs z + y_1$ for all $n \in \N$, so $(T_n x)_{n\in\N}$ is order bounded in $E$.

    (ii) $\Rightarrow$ (iii): Since General Assumption~\ref{general-assumption:standing}\upshape{(a)--(c)} together with (ii) yield General Assumption~\ref{general-assumption:standing} in full, this is an immediate consequence of Theorem~\ref{theorem:banach-principle}\upshape{(i)}.

    The remaining assertion is obvious.
\end{proof}

We close this section with the observation that the order limits yield a well-defined operator on the space of order convergence that is compatible with the topology $\tau$ on $E$.

\begin{proposition} \label{proposition:order-limit-operator}
    Let $X$ be a Banach space and let $E$ be a vector lattice. Suppose that General Assumption~\ref{general-assumption:standing} holds. Consider the convergence set
    \begin{align*}
    	C \coloneqq \set{x \in X : (T_n x)_{n \in \N} \text{ is order convergent in } E}.
    \end{align*}
    Then the operator $Q \colon C \to E$, given by
    \begin{align*}
        Qx = \operatorname*{\mathrm{o}\text{--}\lim}_{n\to\infty} T_n x
    \end{align*}
    is continuous with respect to $\tau$.
\end{proposition}

\begin{proof}
    Clearly, $Q$ is a well-defined linear operator. Let $x \in C$. Since the modulus is order continuous, $\abs{T_n x} \to \abs{Qx}$ in order as $n \to \infty$. As $\abs{T_n x} \leq T^\star x$ for all $n \in \N$, this implies $\abs{Qx} \leq T^\star x$. Consequently, $\rho(Qx) \leq \rho(T^\star x)$ for each $x \in C$ by monotonicity. So, it follows from Lemma~\ref{lemma:maximal-operator-continuity} that $Q$ is continuous at $0$ and therefore continuous.
\end{proof}

\begin{remark} \label{remark:noncommutative}
	A noncommutative counterpart of Theorem~\ref{theorem:banach-principle} is known: Goldstein and Litvinov \cite[Theorem~2]{Goldstein2000} proved a Banach principle in the space of $\tau$-measurable operators affiliated with a semi-finite von Neumann algebra, endowed with the measure topology. The structural ingredients of their proof match General Assumption~\ref{general-assumption:standing} closely, which we regard as evidence that the assumption isolates the correct axiomatics.
\end{remark}

\section{An abstract individual ergodic theorem} \label{section:an-abstract-individual-ergodic-theorem}

In this section, we demonstrate that the abstract framework of Section~\ref{section:an-abstract-banach-principle} yields an individual ergodic theorem on order continuous Banach lattices with a weak unit, in which the sole analytic hypothesis is order boundedness of the Cesàro means in the universal completion. Our argument is entirely lattice-theoretic. The relevant order-theoretic ambient space is the universal completion.

\subsection*{The universal completion and unbounded order convergence}

Let $E$ be an Archimedean vector lattice. $E$ is called \emph{laterally complete} if every non-empty pairwise disjoint family in $E_+$ has a supremum (cf.\ \cite[Definition~7.1]{Aliprantis2003}). Moreover, $E$ is called \emph{universally complete} if it is Dedekind complete and \emph{laterally complete} (see \cite[Definition~7.18]{Aliprantis2003}). The following result is well-established in the theory of vector lattices (see \cite[Theorem~7.21]{Aliprantis2003}).

\begin{proposition} \label{proposition:existence-universal-completion}
    Let $E$ be an Archimedean vector lattice. Then there exists a universally complete vector lattice $E^u$ containing $E$ as an order dense sublattice. $E^u$ is unique up to a Riesz isomorphism.
\end{proposition}

The vector lattice $E^u$ from the proposition above is called the \emph{universal completion} of $E$. The order density of $E$ in its universal completion $E^u$ has several direct consequences that we will frequently make use of throughout this section. We omit the elementary proofs of these statements.

\begin{lemma} \label{lemma:universal-completion-properties}
    Let $E$ be an Archimedean vector lattice and $E^u$ its universal completion. Then the following assertions hold:
    \begin{enumerate}[\upshape (i)]
        \item $E$ is a regular sublattice of $E^u$, i.e., suprema and infima of subsets of $E$ computed in $E$ agree with those computed in $E^u$.
        \item If $E$ is Dedekind $\sigma$-complete, then $E$ is an ideal in $E^u$. In particular, one has $[0,e]_{E^u} \subseteq E$ for each $e \in E_+$.
    \end{enumerate}
\end{lemma}

Recall that a sequence $(x_n)_{n \in \N}$ in a vector lattice $E$ is called \emph{unbounded order convergent} (or \emph{uo-convergent}) to $x \in E$, written $x_n \uoto x$, if $\abs{x_n - x} \wedge u \oto 0$ for each $u \in E_+$. We collect the relevant relations between these notions of order and unbounded order convergence in $E$ and in $E^u$ in the following lemma.

\begin{lemma} \label{lemma:uo-vs-o}
	Let $E$ be a Dedekind $\sigma$-complete vector lattice and let $(x_n)_{n \in \N}$ be a sequence in $E$ with limit $x \in E$. Then the following assertions hold.
	\begin{enumerate}[\upshape (i)]
		\item If $x_n \oto x$ in $E$, then $x_n \oto x$ in $E^u$.
		\item If $x_n \oto x$ in $E^u$, then $x_n \uoto x$ in $E$.
		\item If $(x_n)_{n \in \N}$ is order bounded in $E$, the following assertions are equivalent:
        \begin{enumerate}[\upshape (a)]
            \item $x_n \oto x$ in $E$.
            \item $x_n \oto x$ in $E^u$.
            \item $x_n \uoto x$ in $E$.
        \end{enumerate}
	\end{enumerate}
\end{lemma}

\begin{proof}
	(i) This is an immediate consequence of the regularity of $E$ in $E^u$ (see Lemma~\ref{lemma:universal-completion-properties}).

	(ii) Let $y_n \downarrow 0$ in $E^u$ with $\abs{x_n - x} \leq y_n$ for each $n \in \N$. Fix $u \in E_+$. Then $\abs{x_n - x} \wedge u \leq y_n \wedge u$ and the sequence $(y_n \wedge u)_{n \in \N}$ is contained in $[0,u]_{E^u} \subseteq E$ by the ideal property (cf.~Lemma~\ref{lemma:universal-completion-properties}). Moreover, $y_n \wedge u \downarrow 0$ in $E^u$, hence in $E$ by regularity. Thus $\abs{x_n - x} \wedge u \oto 0$ in $E$.

	(iii) The equivalence (a) $\Leftrightarrow$ (b) follows from (i) and the regularity of $E$ in $E^u$. To see that (b) $\Rightarrow$ (c), note that order convergence in $E$ always implies uo-convergence. Conversely, for (c) $\Rightarrow$ (b) note that if $(x_n)_{n \in \N}$ is order bounded in $E$, say $\abs{x_n} \leq a$ for some $a \in E$ and all $n \in \N$, and $x_n \uoto x$ in $E$, then $\abs{x_n - x} \leq \abs{x_n} + \abs{x} \leq 2a$ and hence, $\abs{x_n - x} = \abs{x_n - x} \wedge 2a \oto 0$.
\end{proof}

\begin{remark} \label{remark:uo-strictly-weaker}
	The converse of Lemma~\ref{lemma:uo-vs-o}(ii) fails in general. For instance, in $E = L^1(0,1)$ with weak unit $e = \one$, the sequence $x_n = n \, \one_{(0,1/n)}$ uo-converges to $0$ in $E$ but $(x_n)_{n \in \N}$ is not order bounded and hence not order convergent in $E^u = L^0(0,1)$ (cf.\ Proposition~\ref{proposition:l0-universal-completion}). This example also illustrates that order convergence in $E^u$ is strictly stronger than uo-convergence in $E$.
\end{remark}

\subsection*{The gauge metric}

Throughout this subsection, let $E$ be a Banach lattice with order continuous norm and a weak unit $e \in E_+$. Recall that the norm of $E$ is \emph{order continuous} if $x_\alpha \downarrow 0$ in $E$ implies $\norm{x_\alpha} \to 0$ and that $e \in E_+$ is called a weak unit if $x \wedge e = 0$ implies $x = 0$ for all $x \in E_+$. Note that each weak unit in a Banach lattice with order continuous norm is already a quasi-interior point, i.e., the principal ideal $E_e \coloneqq \set{x \in E : \abs x \leq \lambda e \text{ for some } \lambda > 0}$ is dense in $E$. As it turns out, under these assumptions, one can construct a strictly positive functional on $E$.

\begin{lemma} \label{lemma:strictly-positive-functional}
    Let $E$ be a Banach lattice with order continuous norm and a weak unit $e \in E_+$. Then there exists a strictly positive functional $\varphi\in E_+'$, i.e., $\varphi(x) > 0$ for all $0 < x \in E_+$.
\end{lemma}

\begin{proof}
By \cite[Theorem~4.15]{Aliprantis2006}, there exists a positive functional $\varphi\in E_+'$ such that $\varphi(x) > 0$ for all $0 < x \in [0, e]$. Let $0 < x \in E_+$. Since $e$ is a weak unit, one has $0 < x \wedge e \leq e$. Thus, $\varphi(x \wedge e) > 0$. As $ 0 \leq x \wedge e \leq x$, it follows $\varphi(x) \geq \varphi(x \wedge e) > 0$. So $\varphi$ is strictly positive.
\end{proof}

Fix some $\varphi \in E'_+$ strictly positive as supplied by Lemma~\ref{lemma:strictly-positive-functional} and consider the \emph{gauge metric} of $\varphi$ on $E^u$, given by
\begin{align*}
    d_\varphi \colon E^u \times E^u \to [0, \infty), \quad d_\varphi(f, g) = \varphi(\abs{f - g} \wedge e).
\end{align*}
Note that this mapping is well-defined, since one has $\abs{f - g} \wedge e \in [0, e]_{E^u} \subseteq E$ by the ideal property (cf.~Lemma~\ref{lemma:universal-completion-properties}(ii)). The next theorem collects some crucial properties of this metric. We note that the assertions~(i) and (iii) were already proven in \cite[Lemma~2.1]{Cerreia2022}. For the convenience of the reader we include the simple proofs nonetheless.

\begin{proposition} \label{proposition:gauge-metric}
    Let $E$ be a Dedekind $\sigma$-complete Banach lattice with order continuous norm and a weak unit $e \in E_+$. Then the following assertions hold:
    \begin{enumerate}[\upshape (i)]
        \item $d_\varphi$ is a metric on $E^u$.
        \item The canonical embedding $\iota \colon E \to E^u$ is continuous with respect to $d_\varphi$.
        \item The topology generated by $d_\varphi$ is locally solid.
        \item The topology generated by $d_\varphi$ has the $\sigma$-Lebesgue property.
    \end{enumerate}
\end{proposition}

\begin{proof}
    (i): The definiteness follows from the strict positivity of $\varphi$ and the fact that $e$ is a weak order unit in $E^u$. Indeed, if $d_\varphi(f,g)=0$, then $\varphi(\abs{f-g}\wedge e)=0$, so strict positivity yields $\abs{f-g}\wedge e=0$; since $e$ is a weak unit in $E^u$, this implies $f=g$. The triangle inequality is a consequence of the general Riesz space identity $(x + y) \wedge z \leq (x \wedge z) + (y \wedge z)$ for all $x, y, z \in E_+$. The other properties are obvious.

    (ii): Let $x, y \in E$. Then one has
    \begin{align*}
        d_\varphi(\iota x, \iota y) = \varphi(\abs{x - y} \wedge e) \leq \norm{\varphi} \cdot \norm{\abs{x - y} \wedge e} \leq \norm{\varphi} \cdot \norm{x - y},
    \end{align*}
    which shows that $\iota$ is even Lipschitz-continuous.

    (iii): Fix $\varepsilon > 0$ and consider the open ball $\mathrm B_\varepsilon = \set{x \in E^u : \varphi(\abs x \wedge e) < \varepsilon}$. Let $x \in \mathrm B_\varepsilon$ and $y \in E^u$ such that $\abs y \leq \abs x$. Then $\varphi(\abs y \wedge e) \leq \varphi(\abs x \wedge e) \leq \varepsilon$, i.e., $y \in \mathrm B_\varepsilon$. Thus, $\mathrm B_\varepsilon$ is solid. Since the open balls around the origin are a neighbourhood base of $0$, it follows that the topology generated by $d_\varphi$ is locally solid.

    (iv): Let $(f_n)_{n \in \N}$ in $E^u$ such that $f_n\downarrow 0$ in $E^u$. Then $f_n \wedge e \downarrow 0$ in $E^u$ and by the ideal property (cf.~Lemma~\ref{lemma:universal-completion-properties}(ii)) one has $f_n\wedge e \in E$ since $f_n \wedge e \leq e$ for all $n \in \N$.
    As $E$ is Dedekind $\sigma$-complete, it follows that $f_n\wedge e \downarrow 0$ in $E$. From the order continuity of the norm it follows that $\norm{f_n \wedge e} \to 0$. Thus,
    \begin{align*}
        0 \leq d_\varphi(f_n,0) = \varphi(f_n \wedge e) \leq \norm{\varphi} \cdot \norm{f_n \wedge e} \to 0
    \end{align*}
    as $n \to \infty$. Thus, the topology generated by $d_\varphi$ has the $\sigma$-Lebesgue property.
\end{proof}

The above proposition shows that under certain conditions on a Riesz space one can define a topology on its universal completion. It is unclear to the author if the universal completion can be equipped with a topology in general.

\subsection*{Individual ergodic theorem}

Using the Banach principle from Section~\ref{section:an-abstract-banach-principle} we are able to prove an abstract individual ergodic theorem on order continuous Banach lattices with a weak unit. Recall that a positive operator $T$ on a Banach space is called \emph{power-bounded} if $\sup_{n \in \N} \norm{T^n} < \infty$, and \emph{mean ergodic} if its Cesàro means $A_n \coloneqq \frac{1}{n} \sum_{k=0}^{n-1} T^k$ converge strongly.

\begin{theorem} \label{theorem:individual-ergodic}
	Let $E$ be an order continuous Banach lattice with a weak unit $e \in E_+$. Let $T$ be a positive, power-bounded, mean ergodic operator on $E$ with mean ergodic projection $P$. Suppose that the following assertions hold:
	\begin{enumerate}[\upshape (i)]
		\item The weak unit is a super fixed point of $T$, i.e., $Te \leq e$.
		\item For each $x \in E$ the sequence $(A_n \abs{x})_{n \in \N}$ is order bounded in $E^u$.
	\end{enumerate}
	Then, for each $x \in E$, the sequence $(A_n x)_{n \in \N}$ is order convergent to $Px$ in $E^u$ and, in particular, $A_n x \uoto Px$ in $E$.
\end{theorem}

As a first step towards the proof of Theorem~\ref{theorem:individual-ergodic} we prove the following lemma that provides a dense subspace on which the Cesàro means converge to the mean ergodic projection with respect to the order in $E^u$.

\begin{lemma} \label{lemma:dense-class}
	Let $T$ be a positive, power bounded, mean ergodic operator on a Banach lattice with order continuous norm with mean ergodic projection $P$ and suppose that $Te \leq e$ for some weak unit $e \in E_+$. Then the subspace
    \begin{align*}
        D \coloneqq \fix T \oplus (I - T) E_e
    \end{align*}
    is norm dense in $E$ and $A_n x \oto Px$ for each $x \in D$.
\end{lemma}

\begin{proof}
	As $E$ is a Banach lattice, it follows that $e$ is a quasi-interior point of $E$ (see \cite[Theorem~4.85]{Aliprantis2006}). Thus, the principal ideal $E_e$ is norm dense in $E$. Since $I - T$ is bounded, it follows that $(I - T) E_e$ is dense in $\overline{(I - T) E}$. As $T$ is mean ergodic, by \cite[Theorem~8.20]{Eisner2015}, one has the mean ergodic decomposition
    \begin{align*}
        E = \fix T \oplus \overline{(I - T) E}.
    \end{align*}
    Consequently, $D$ is norm dense in $E$.

	To prove the statement on the order convergence, let $x = u + (I - T) y$ with $u \in \fix T$ and $y \in E_e$. Then there exists $c > 0$ such that $\abs y \leq c e$. Moreover, $\abs{T^n y} \leq T^n \abs{y} \leq c T^n e \leq c e$ and thus,
	\begin{align*}
		\abs{A_n x - Px} = \abs{A_n (I - T) y} = \frac{1}{n} \abs{y - T^n y} \leq \frac{1}{n} (\abs{y} + \abs{T^n y}) \leq \tfrac{2c}{n} e \downarrow 0
	\end{align*}
    by the Archimedean property. Thus, $A_n x \oto Px$ in $E$ and thus, in $E^u$.
\end{proof}

We are now in position to prove the main theorem of this section.

\begin{proof}[Proof~of~Theorem~\ref{theorem:individual-ergodic}]
    In what follows we write $M$ for the maximal operator associated with the sequence $(A_n)_{n\in\N}$. Consider the Cesàro means
    \begin{align*}
        A_n \colon E \to E^u, \quad A_n x = \frac{1}{n} \sum_{k = 0}^{n - 1} T^k x
    \end{align*}
    regarded as an element of $E^u$ via the order dense embedding $E \hookrightarrow E^u$. As $E$ is continuously embedded into $E^u$ by Lemma~\ref{lemma:universal-completion-properties}(ii), one has $A_n \in \calL(E; E^u)$ for all $n \in \N$. Note further that $E^u$ is Dedekind complete and that assertion (ii) shows that the Cesàro means are pointwise order bounded in $E^u$. Moreover, the topology generated by the gauge metric on $E^u$ is locally solid and has the $\sigma$-Lebesgue property by Proposition~\ref{proposition:gauge-metric}(iii)~and~(iv). Thus, the General Assumption~\ref{general-assumption:standing} is satisfied and we can consider the maximal operator
    \begin{align*}
    	M \colon E \to E^u_+, \quad M x \coloneqq \sup_{n \in \N} A_n \abs{x}
    \end{align*}
    and employ the abstract Banach principle (see Theorem~\ref{theorem:banach-principle}). Due to assumption~(i) and Lemma~\ref{lemma:dense-class} this yields that $(A_n x)_{n \in \N}$ is order convergent in $E^u$ for each $x \in E$. By Proposition~\ref{proposition:order-limit-operator} there exists an operator $Q \colon E \to E^u$ such that $Qx = \operatorname*{\mathrm{o}\text{--}\lim}_{n\to\infty} A_n x$.

    It remains to show that $Q E\subseteq E$ and that $Q=P$. By mean ergodicity, one has $A_n x \to Px$ in norm for each $x\in E$. By the continuity of the embedding $E\hookrightarrow E^u$, this implies $A_nx\to Px$ with respect to the gauge metric $d_\varphi$. On the other hand, $A_n x \oto Q x$ in $E^u$. Since the topology induced by the gauge metric has the $\sigma$-Lebesgue property, it follows that $A_n x \to Qx$ with respect to $d_\varphi$. By uniqueness of limits this shows $Qx = Px$. In particular, $Qx\in E$ and therefore $QE \subseteq E$. Thus $Q=P$.
    Finally, Lemma~\ref{lemma:uo-vs-o}(ii) shows that $A_n x \uoto Px$ for each $x \in E$.
\end{proof}

\section{An abstract maximal ergodic theorem} \label{section:an-abstract-maximal-ergodic-theorem}

The individual ergodic theorem of Section~\ref{section:an-abstract-individual-ergodic-theorem} reduces the problem of pointwise ergodic convergence to an order boundedness condition for the Cesàro means in the universal completion. In this section, we establish this maximal hypothesis under natural superinvariance assumptions. The proof is based on a lattice-theoretic version of Hopf's maximal inequality due to Garsia's argument. Here, indicator functions are replaced by band projections, while the integral is replaced by a strictly positive order continuous functional. The passage from the resulting level-band estimate to order boundedness is then carried out entirely within the universal completion.

\subsection*{Superinvariant pairs}

Throughout this section, let $E$ be a Dedekind complete vector lattice and let $T$ be a positive operator on $E$. In what follows we write $S_n \coloneqq \sum_{k=0}^{n-1} T^k$ and $A_n \coloneqq \frac{1}{n} S_n$, $n \in \N$. We first discuss the lattice-theoretic form of Hopf's inequality.

\begin{lemma} \label{lemma:hopf-inequality}
    Let $\varphi \in E_+'$ be a super fixed point of $T'$. Consider the operators
    \begin{align*}
        M_N \colon E \to E, \quad M_N x = \max_{1 \leq n \leq N} S_n x, \quad N \in \N.
    \end{align*}
    Let $x \in E$ and let $P_N$ be the band projection onto the band $\set{(M_Nx)^+}^{\perp \perp}$. Then
    \begin{align*}
        \dual{P_Nx, \varphi} \geq 0.
    \end{align*}
\end{lemma}

\begin{proof}
    Note that for $1 \leq n \leq N$ one has
    \begin{align*}
        S_{n-1} x \leq (M_Nx)^+.
    \end{align*}
    Indeed, $S_{n-1} x \leq M_N x \leq (M_Nx)^+$. Since $T$ is positive,
    \begin{align*}
    S_nx = x + TS_{n-1} x \leq x + T(M_Nx)^+.
    \end{align*}
    Taking the maximum yields
    \begin{align*}
    M_Nx - T(M_Nx)^+ \leq x.
    \end{align*}
    Now, we apply the positive operator $P_N$ to this inequality. Since $(M_N x)^-$ is disjoint to $(M_N x)^+$, we obtain $P_N M_N x = (M_N x)^+$, and $0 \leq P_N \leq I$ yields $P_N T (M_N x)^+ \leq T (M_N x)^+$. Consequently, $P_N x \geq (M_N x)^+ - T (M_N x)^+$ and therefore
    \begin{align*}
    \dual{P_Nx, \varphi} \geq \dual{(M_Nx)^+, \varphi} -  \dual{(M_Nx)^+, T'\varphi} \geq 0,
    \end{align*}
    since $\varphi$ is a super fixed point of $T'$.
\end{proof}

As it turns out, the preceding inequality becomes a maximal estimate once $T$ also has a super invariant fixed point. This motivates the following terminology: Let $T$ be a positive operator on $E$. We call a pair $(e, \varphi)$ a \emph{superinvariant pair} for $T$ if the following assertions hold:
\begin{enumerate}[\upshape (a)]
    \item $e\in E_+$ is a weak unit and a super fixed point of $T$, i.e., $Te\leq e$.
    \item $\varphi \in E'_+$ is a strictly positive, $\sigma$-order continuous functional which is a super fixed point of $T'$, i.e., $T' \varphi \leq \varphi$.
\end{enumerate}

In our abstract setting, the weak unit $e$ will play the role of the constant one function, while the functional $\varphi$ will take the role of integration. The following proposition is the corresponding abstract weak type estimate.

\begin{proposition} \label{proposition:maximal-level-estimate}
    Let $T$ be a positive operator on $E$. Suppose that $T$ admits a superinvariant pair $(e,\varphi)$. Let $P_{N,\lambda}$ be the band projection onto the band generated by
    \begin{align*}
    \sup_{1\leq n\leq N}(A_nx-\lambda e)^+, \quad N \in \N.
    \end{align*}
    Then one has $\lambda \dual{P_{N,\lambda}e, \varphi} \leq \dual{x, \varphi}$ for all $\lambda > 0$ and $x \in E_+$.
\end{proposition}

\begin{proof}
    Fix $\lambda > 0$ and $x \in E_+$. Set $y \coloneqq x - \lambda e$. As $e$ is a super fixed point of $T$, one has $S_n e\leq ne$ and thus
    \begin{align*}
    S_n y = S_n x - \lambda S_n e \geq n (A_n x - \lambda e)
    \end{align*}
    for all $n \in \N$. Consequently,
    \begin{align*}
    (M_Ny)^+ \geq (S_ny)^+ \geq n(A_nx-\lambda e)^+ \geq (A_nx-\lambda e)^+
    \end{align*}
    for all $1 \leq n \leq N$. If $B_N \coloneqq \set{(M_Ny)^+}^{\perp \perp}$ and
    \begin{align*}
        C_{N,\lambda}
        \coloneqq
        \bigg(\sup_{1\leq n\leq N} (A_nx - \lambda e)^+\bigg)^{\perp \perp},
    \end{align*}
    then this implies that $C_{N,\lambda}\subseteq B_N$ and therefore
    \begin{align*}
    P_{N,\lambda} e \leq P_{B_N} e
    \end{align*}
    for the associated band projections. Applying Lemma~\ref{lemma:hopf-inequality} to $y$ yields
    \begin{align*}
    0 \leq \dual{P_{B_N}y, \varphi} &= \dual{P_{B_N}x, \varphi} - \lambda \dual{P_{B_N}e, \varphi} \\
    &\leq \dual{x, \varphi} - \lambda \dual{P_{N, \lambda}e, \varphi},
    \end{align*}
    which shows the claim.
\end{proof}

We next pass to the infinite maximal function. For $x\in E_+$ and $\lambda>0$, consider the band
\begin{align*}
C_\lambda(x) \coloneqq \set{(A_nx-\lambda e)^+ : n\in\N}^{\perp \perp},
\end{align*}
and let $Q_\lambda(x)$ be the corresponding band projection. The monotonicity of the finite level bands allows us to pass to the limit.

\begin{lemma} \label{lemma:maximal-level-band}
    Let $T$ be a positive operator on $E$. Suppose that $T$ admits a superinvariant pair $(e,\varphi)$. Then, for every $x\in E_+$ and $\lambda>0$,
    \begin{align*}
    Q_\lambda(x)e
    =
    \sup_{N\in\N}P_{N,\lambda}e
    \end{align*}
    and
    \begin{align} \label{eq:maximal-weak-type}
        \lambda \dual{Q_\lambda(x)e, \varphi} \leq \dual{x, \varphi}.
    \end{align}
\end{lemma}

\begin{proof}
    The bands $C_{n, \lambda}$ form an increasing sequence whose union generates $C_\lambda(x)$. Hence, by monotone continuity of band projections, one has
    \begin{align*}
        P_{N,\lambda}e\uparrow Q_\lambda(x)e.
    \end{align*}
    Since $\varphi$ is $\sigma$-order continuous, Proposition~\ref{proposition:maximal-level-estimate} yields
    \begin{align*}
        \lambda \dual{Q_\lambda(x)e, \varphi} =\sup_{N\in\N} \lambda \dual{P_{N,\lambda}e, \varphi} \leq \dual{x, \varphi},
    \end{align*}
    which proves the claim.
\end{proof}

The weak type estimate now implies that the Cesàro orbit is order bounded in the universal completion. The following lemma is the lattice-theoretic substitute for the classical assertion that a measurable function whose distribution function tends to zero at infinity is finite almost everywhere.

\begin{lemma} \label{lemma:lateral-boundedness}
    Let $E$ be a Dedekind complete vector lattice with a weak unit $e \in E$, and let $A \subseteq E_+$. Let $Q_m$, $m \in \N$ denote the band projection onto the band
    \begin{align*}
        B_m \coloneqq \set{(a-me)^+ : a\in A}^{\perp \perp}.
    \end{align*}
    If $Q_m e \downarrow 0$, then $A$ is order bounded in $E^u$.
\end{lemma}

\begin{proof}
    By the ideal property of $E$ in $E^u$ (cf.\ Lemma~\ref{lemma:universal-completion-properties}(ii)), all bands and band projections generated by subsets of $E$ may be computed in $E^u$ without changing their action on $E$. We may therefore regard the sequence $(Q_m)_{m\in\N}$ as a decreasing sequence of band projections on $E^u$.

    Let $B_m$ denote the band generated by
    \begin{align*}
    \set{(a-me)^+ : a\in A}.
    \end{align*}
    Then $(B_m)_{m\in\N}$ is decreasing. Let $B_\infty \coloneqq \bigcap_{m\in\N} B_m$ and let $Q_\infty$ be the corresponding band projection (cf.\ \cite[Theorem~1.2.10]{MeyerNieberg1991}). By monotonicity of band projections,
    \begin{align*}
    Q_\infty e =\inf_{m \in \N} Q_me =0.
    \end{align*}
    Let $0 \leq x \in B_\infty$. Since $e$ is a weak unit in $E^u$, one has $x \wedge ne \uparrow x$ as $n \to \infty$ and $0 \leq x \wedge ne = Q_\infty(x \wedge ne) \leq n Q_\infty e = 0$ for all $n \in \N$. Thus, $B_\infty=\set{0}$.

    Now set $D_m\coloneqq Q_{m-1}-Q_m$, $m \in \N$ and $Q_0 = I$. The $D_m$ are pairwise disjoint band projections. For $a \in A$, one has $D_m(a-me)^+=0$ and consequently
    \begin{align*}
    D_m a =D_m (a - me) + m D_m e \leq m D_m e.
    \end{align*}
    Moreover, the family $(m D_m e)_{m \in \N}$ is pairwise disjoint and positive. As $E^u$ is laterally complete, the element $u \coloneqq \sup_{m \in \N} m D_m e$ exists in $E^u$. Thus, for each $M \in \N$ one has
    \begin{align*}
    (I - Q_M) a =\sum_{m = 1}^M D_m a \leq \sum_{m = 1}^M m D_m e \leq u.
    \end{align*}
    As $Q_M a \downarrow Q_\infty a = 0$, it follows that
    \begin{align*}
    a= \sup_{M \in \N} (I - Q_M)a \leq u.
    \end{align*}
    Thus $A$ is order bounded in $E^u$.
\end{proof}

\subsection*{Maximal ergodic theorem}

We can now combine the weak type estimate with the preceding lemma. This gives the maximal hypothesis required in the abstract Banach principle. In the statement and proof we write
\begin{align*}
    M x \coloneqq \sup_{n \in \N} A_n \abs{x} \in E^u
\end{align*}
for the \emph{maximal operator} associated with the sequence $(A_n)_{n\in\N}$; by assertion~(i) below this supremum exists in $E^u$ for each $x \in E$. We start by proving a preparatory lemma.

\begin{lemma} \label{lemma:phi-decay-to-order-limit}
    Let $E$ be a Dedekind complete vector lattice and let $\varphi \in E'_+$ be strictly positive. If $(y_m)_{m \in \N}$ is a decreasing sequence in $E_+$ such that $\dual{y_m, \varphi} \to 0$, then $y_m \downarrow 0$ as $m \to \infty$.
\end{lemma}

\begin{proof}
    Since $(y_m)_{m \in \N}$ is decreasing and bounded below by $0$, the infimum $y \coloneqq \inf_{m \in \N} y_m$ exists in $E$ by Dedekind completeness. As $0 \leq y \leq y_m$ for each $m \in \N$, positivity of $\varphi$ yields
    \begin{align*}
        0 \leq \dual{y, \varphi} \leq \dual{y_m, \varphi} \to 0.
    \end{align*}
    Thus, $\dual{y, \varphi} = 0$ and the strict positivity of $\varphi$ implies $y = 0$.
\end{proof}

Combined with Lemma~\ref{lemma:lateral-boundedness}, this shows that any quantitative decay $\varphi(Q_m e) \to 0$ of a strictly positive functional along the level bands already forces order boundedness in $E^u$. We are now in position to prove an abstract maximal ergodic theorem.

\begin{theorem}[Maximal ergodic theorem] \label{theorem:maximal-ergodic}
    Let $E$ be a Dedekind complete vector lattice and let $T$ be a positive operator on $E$ which admits a superinvariant pair $(e,\varphi)$. Then the following assertions hold:
    \begin{enumerate}[\upshape (i)]
        \item For each $x \in E$, the set $\set{A_n \abs{x} : n \in \N}$ is order bounded in $E^u$.
        \item For each $x \in E$ and $\lambda > 0$, let $q_\lambda(x)$ denote the projection of $e$ onto the band generated by $(M \abs{x} - \lambda e)^+$ in $E^u$. Then $q_\lambda(x)\in E$ and
        \begin{align} \label{eq:maximal-ergodic-weak-type}
        	\lambda\,\varphi(q_\lambda(x))
        	\leq\varphi(\abs{x}).
        \end{align}
    \end{enumerate}
\end{theorem}

\begin{proof}
    (i): Fix $x\in E$ and set $a_n\coloneqq A_n\abs{x}$. For $\lambda>0$, consider the band $C_\lambda \coloneqq \set{(a_n-\lambda e)^+:n\in\N}^{\perp \perp}$ and let $Q_\lambda$ be the corresponding band projection. By Lemma~\ref{lemma:maximal-level-band}, one has
    \begin{align*}
        m \dual{Q_m(x) e, \varphi} \leq \dual{\abs{x}, \varphi} \qquad (m \in \N).
    \end{align*}
    The sequence $(Q_m(x) e)_{m \in \N}$ is decreasing. Set $y \coloneqq \inf_{m\in\N}Q_m(x) e$. Then it follows from Lemma~\ref{lemma:phi-decay-to-order-limit} that $y = 0$. Lemma~\ref{lemma:lateral-boundedness}, applied to the set $A=\set{a_n : n\in\N}$ now shows that $(a_n)_{n\in\N}$ is order bounded in $E^u$. Hence, $Mx$ exists in $E^u$, and since $T$ is positive, one has
    \begin{align*}
        \abs{A_n x} \leq A_n \abs{x} \leq Mx.
    \end{align*}
    Thus, $\set{A_n \abs{x} : n \in \N}$ is order bounded in $E^u$.

    (ii): We show that $Q_\lambda(x)e$ coincides with the projection of $e$ onto the band generated by $(Mx-\lambda e)^+$ in $E^u$. Since $(a_n-\lambda e)^+ \leq (Mx-\lambda e)^+$ for each $n \in \N$, the band generated by the former elements is contained in the band generated by $(Mx - \lambda e)^+$. Conversely, one has
    \begin{align*}
    Mx-\lambda e
    =\sup_{n\in\N}(a_n-\lambda e)
    \leq\sup_{n\in\N}(a_n-\lambda e)^+,
    \end{align*}
    and the supremum on the right is contained in the band generated by the $(a_n-\lambda e)^+$. Hence $(Mx-\lambda e)^+$ belongs to the same band. So the two bands coincide, and \eqref{eq:maximal-ergodic-weak-type} follows from Lemma~\ref{lemma:maximal-level-band}.
\end{proof}

\begin{remark} \label{remark:maximal-ergodic-classical}
    Theorem~\ref{theorem:maximal-ergodic} is a lattice-theoretic version of Hopf's maximal ergodic theorem (see \cite[Theorem~11.1]{Eisner2015}). In the classical setting $E=L^1(\Omega)$, with $e=\one$ and
    \begin{align*}
    \varphi(f)=\int_\Omega f \,\ud\mu,
    \end{align*}
    the band projection associated with a measurable set $B$ is given by the multiplication by its indicator function. Consequently,
    \begin{align*}
    \varphi(q_\lambda(f)) = \mu (\set{M f > \lambda}),
    \end{align*}
    and \eqref{eq:maximal-ergodic-weak-type} becomes the usual weak type $(1,1)$-estimate
    \begin{align*}
    \lambda \mu (\set{M f > \lambda})
    \leq \norm{f}_1.
    \end{align*}
    The proof above, however, does not use a measure representation. The band projection $q_\lambda(f)$ is the intrinsic lattice-theoretic substitute for the level set $\set{Mf > \lambda}$.
\end{remark}

\begin{remark} \label{remark:calderon}
	A superinvariant pair encodes an abstract Dunford--Schwartz structure. Indeed, $\norm{z}_\varphi \coloneqq \varphi(\abs{z})$ defines a lattice norm on $E$ which is additive on $E_+$, so the completion of $(E, \norm{\, \cdot \,}_\varphi)$ is an AL-space, while $(E_e, \norm{\, \cdot \,}_e)$ is an AM-space with unit $e$; the pair conditions assert precisely that $T$ acts contractively on both. In the Kakutani representations, $T$ therefore extends to a positive operator which contracts an $L^1$-space and an $L^\infty$-space simultaneously, and Theorem~\ref{theorem:maximal-ergodic} may be viewed as the intrinsic, representation-free formulation of Hopf's theorem for this couple.
\end{remark}

The main consequence of the preceding individual ergodic theorem is immediate. The maximal theorem supplies precisely the order boundedness assumption appearing there.

\begin{corollary} \label{corollary:individual-from-maximal}
    Let $E$ be an order continuous Banach lattice with a weak unit $e\in E_+$. Let $T$ be a positive, power-bounded, mean ergodic operator on $E$ with mean ergodic projection $P$. Suppose that $(e,\varphi)$ is a superinvariant pair for $T$ with $\varphi\in E'_+$, so that in particular $\varphi$ is $\sigma$-order continuous. Then $A_nx\uoto Px$ for each $x\in E$.
\end{corollary}

\begin{proof}
    The sequence $(A_n\abs{x})_{n\in\N}$ is order bounded in $E^u$ for each $x \in E$ by Theorem~\ref{theorem:maximal-ergodic}. Thus, the maximal hypothesis of Theorem~\ref{theorem:individual-ergodic} is satisfied and the assertion follows.
\end{proof}

\begin{remark}
    The preceding corollary separates the two ingredients of the individual ergodic theorem. The Banach principle of Section~\ref{section:an-abstract-banach-principle} provides the abstract mechanism which upgrades convergence on a dense class to convergence on the whole space, while Theorem~\ref{theorem:maximal-ergodic} supplies the required maximal hypothesis.
\end{remark}

\begin{remark}
    The proof of Theorem~\ref{theorem:maximal-ergodic} also explains the role of the universal completion. The weak type estimate only yields that the bands $Q_me$ vanish as $m\to\infty$. This does not by itself produce an order bound in $E$, since the Cesàro means need not be order bounded there. Lateral completeness of $E^u$ allows the disjoint pieces $m(Q_{m-1}-Q_m)e$ to be assembled into a single dominating element. Thus the universal completion is not merely a convenient ambient space: it is precisely the lattice-theoretic substitute for the passage from finiteness of a maximal function almost everywhere to an actual dominating function.
\end{remark}

\subsection*{Superinvariant pairs from Perron--Frobenius theory}

We close this section with a brief discussion of the role that Perron-Frobenius theory might play when it comes to the existence of superinvariant pairs. In particular, we show that superinvariant pairs are not exotic. Recall that an operator $T$ on a Banach lattice $E$ is called \emph{irreducible} if $\set{0}$ and $E$ are the only closed $T$-invariant ideals of $E$.

\begin{proposition} \label{proposition:irreducible}
	Let $T$ be an irreducible operator on a Banach lattice $E$. Then the following assertions hold:
	\begin{enumerate}[\upshape (i)]
    	\item Every non-zero super fixed point of $T$ is a quasi-interior point of $E$.
		\item Every non-zero super fixed point of $T'$ is strictly positive.
	\end{enumerate}
\end{proposition}

\begin{proof}
	(i): Let $0 < e \in E_+$ be a super fixed point of $T$. Then the closure $\overline{E_e}$ is a closed ideal of $E$. Moreover, for each $y \in E_e$ there exists $c > 0$ such that $\abs{y} \leq c e$. Thus, $\abs{Ty} \leq c Te \leq c e$, i.e., $T E_e \subseteq E_e$. As $e \ne 0$, the irreducibility of $T$ yields $\overline{E_e} = E$.

    (ii): Let $0 < \varphi \in E_+'$ be a super fixed point of $T'$. Consider the null ideal $N_\varphi \coloneqq \set{x \in E : \dual{\abs{x}, \varphi} = 0}$. Then $N_\varphi$ is closed and $T$-invariant, since
    \begin{align*}
        \dual{\abs{Tx}, \varphi} \leq \dual{\abs{x}, T' \varphi} \leq \dual{\abs{x}, \varphi}
    \end{align*}
    for each $x \in E$. As $\varphi \neq 0$, one has $N_\varphi \neq E$, so the irreducibility of $T$ yields $N_\varphi = \set{0}$, which is the claim.
\end{proof}

\begin{proposition} \label{proposition:existence}
	Let $T$ be power-bounded and let $e \in E_+ \setminus \set{0}$ satisfy $Te = e$. Then there exists a positive $\varphi \in E'$ with $T' \varphi = \varphi$ and $\varphi(e) > 0$.
\end{proposition}

\begin{proof}
	Choose $\psi_0 \in E'$ with $\norm{\psi_0} = 1$ and $\psi_0(e) = \norm{e}$. Then $\psi \coloneqq \abs{\psi_0} \in E'_+$ satisfies $\psi(e) \geq \psi_0(e) > 0$. Consider
    the Cesàro means $A_n \coloneqq \frac{1}{n} \sum_{k = 0}^{n - 1} T^k$ and set $\Phi_n \coloneqq A_n' \psi \in E'_+$ for $n \in \N$. As $T$ is power-bounded, $(\Phi_n)_{n \in \N}$ is norm bounded, so by the Banach--Alaoglu theorem it admits a weak$^\ast$-cluster point $\varphi \in E'$, which is positive as the dual cone is weak$^\ast$-closed. Since $Te = e$ implies $A_n e = e$, one has $\Phi_n(e) = \psi(e)$ for all $n \in \N$, so $\varphi(e) = \psi(e) > 0$. Finally, since $T$ is power bounded, one has
	\begin{align*}
		\big( T' \Phi_n - \Phi_n \big)(x) = \frac{1}{n} \, \psi\big( T^n x - x \big) \to 0 \qquad (x \in E)
	\end{align*}
	as $n \to \infty$. Since $T'$ is weak$^\ast$-continuous, by passing to the cluster point one obtains $T' \varphi = \varphi$.
\end{proof}

Combining Theorem~\ref{theorem:maximal-ergodic}, Theorem~\ref{theorem:individual-ergodic} and the two propositions above, we arrive at an individual ergodic theorem whose hypotheses are entirely spectral-theoretic.

\begin{corollary} \label{corollary:perron-frobenius}
	Let $E$ be an order continuous Banach lattice with a weak unit and let $T$ be a positive, power-bounded, irreducible and mean ergodic operator on $E$ whose mean ergodic projection $P$ is non-zero. Then $A_n x \uoto Px$ in $E$ for each $x \in E$.
\end{corollary}

\begin{proof}
	Clearly, $P$ is positive, and since $P \neq 0$ there exists $z \in E_+$ with $e \coloneqq Pz \ne 0$. As $P E = \fix T$, one has $Te = e$. By Proposition~\ref{proposition:existence} there exists $0 < \varphi \in E'_+$ such that $T' \varphi = \varphi$ and $\varphi(e) > 0$. So, Proposition~\ref{proposition:irreducible} implies that $e$ is a quasi-interior point and that $\varphi$ is strictly positive. Moreover, $\varphi$ is $\sigma$-order continuous by the order continuity of the norm on $E$. Consequently, $(e, \varphi)$ is a superinvariant pair for $T$, and Corollary~\ref{corollary:individual-from-maximal} concludes the proof.
\end{proof}

\section{Applications and relations to the classical theory} \label{section:applications-and-relations-to-the-classical-theory}

The abstract framework developed in the previous section raises the question of how natural its underlying assumptions are. We shall see that they are precisely those satisfied in the classical setting of measurable functions.

\subsection*{The model space of measurable functions} Let $(\Omega, \Sigma, \mu)$ be a $\sigma$-finite measure space and let $E \coloneqq L^0(\Omega)$ be the vector lattice of all equivalence classes of measurable functions. The following results show that General~Assumption~\ref{general-assumption:standing} is quite natural in this concrete setting. We first record the lattice-theoretic properties of $E$.

\begin{proposition} \label{proposition:l0-lattice-properties}
    Let $(\Omega, \Sigma, \mu)$ be a measure space and let $E \coloneqq L^0(\Omega)$. Then the following assertions hold:
    \begin{enumerate}[\upshape (i)]
        \item $E$ is Archimedean.
        \item $E$ is universally complete.
    \end{enumerate}
\end{proposition}

\begin{proof}
    (i): This is \cite[Example~9.2(iv)]{Zaanen2012}.

    (ii): It is a well-known fact from basic measure theory that $E$ is Dedekind complete. So it is left to show the lateral completeness of $E$. So let $(f_i)_{i \in I}$ be a disjoint family in $E$. Then the essential supports of the $f_i$ are pairwise disjoint modulo null sets. Thus, one has the identity
    \begin{align*}
        \sup_{i \in I} f_i = \sum_{i \in I} f_i \in E
    \end{align*}
    almost everywhere. Hence, $E$ is laterally complete.
\end{proof}

The following proposition shows that $L^0(\Omega)$ is the universal completion of each $L^p$-space, $1 \leq p \leq \infty$ (see \cite[Theorem~7.73]{Aliprantis2003}).

\begin{proposition} \label{proposition:l0-universal-completion}
    Let $(\Omega, \Sigma, \mu)$ be a $\sigma$-finite measure space, $E \coloneqq L^0(\Omega)$ and let $1 \leq p \leq \infty$. Then $E$ is the universal completion of $L^p(\Omega)$.
\end{proposition}

The next result characterizes order convergence and unbounded order convergence in $E$ as well as the notion of order boundedness.

\begin{proposition} \label{proposition:abstract-banach-principle}
    Let $(\Omega, \Sigma, \mu)$ be a $\sigma$-finite measure space and let $E \coloneqq L^0(\Omega)$. Then the following assertions hold:
	\begin{enumerate}[\upshape (i)]
		\item A sequence $(f_n)_{n \in \N}$ in $E$ is order bounded if and only if $\sup_{n \in \N} \abs{f_n} < \infty$ almost everywhere.
		\item A sequence in $E$ is uo-convergent if and only if it converges almost everywhere.
        \item A sequence in $E$ is order convergent if and only if it converges almost everywhere and is order bounded.
	\end{enumerate}
\end{proposition}

\begin{proof}
    (i): Suppose first that $(f_n)_{n\in\mathbb N}$ is order bounded in $E$. Then there exists
    $g\in E_+$ such that $\abs{f_n} \leq g$ for all $n \in \N$. Since $g \in L^0(\Omega)$, one has $ g < \infty$ almost everywhere. Hence,
    $\sup_{n\in\mathbb N} \abs{f_n(x)} \leq g(x) < \infty$ for a.e.\ $x \in\Omega$.
    Therefore, $\sup_{n \in \N} \abs{f_n} < \infty$ almost everywhere.

    Conversely, assume that $g \coloneqq \sup_{n\in\mathbb N} \abs{f_n} < \infty$ almost everywhere. Then clearly $g \in E$ and $\abs{f_n} \leq g$ for all $n \in \N$. So $(f_n)_{n\in\mathbb N}$ is order bounded in $E$.

    (ii): This is \cite[Proposition~3.1]{Gao2017}.

    (iii): First, let $(f_n)_{n\in\mathbb N}$ in $E$ and $f \in E$ such that $f_n \to f$ in order as $n \to \infty$. Then there exists a sequence $(u_n)_{n\in\mathbb N}$ in $E_+$ such that $u_n \downarrow 0$ and $\abs{f_n - f} \leq u_n$ for all $n \in \N$. Since $u_n \downarrow 0$, there exists a null set $N \subseteq \Omega$ such that $u_n(x) \downarrow 0$ for all $x \in \Omega \setminus N$. Consequently, $\abs{f_n(x) - f(x)} \leq u_n(x) \to 0$ as $n \to \infty$ for all $x \in \Omega \setminus N$. Thus, $f_n \to f$ almost everywhere. Moreover,
    \begin{align*}
        \abs{f_n} \leq \abs{f} + \abs{f_n - f} \leq \abs{f} + u_n \leq \abs{f} + u_1.
    \end{align*}
    Thus $(f_n)$ is order bounded.

    Conversely, suppose that $f_n \to f$ almost everywhere and that $(f_n)$ is
    order bounded. Then there exists $g \in E_+$ such that $\abs{f_n} \leq g$ for all $n \in \N$. Passing to the limit almost everywhere yields $\abs{f} \leq g$ almost everywhere, and therefore $\abs{f_n - f} \leq \abs{f_n} + \abs{f} \leq 2g$. Define $u_N \coloneqq \sup_{n \geq N} \abs{f_n - f} \in E$. Then $u_N \downarrow 0$ in $E$ and $\abs{f_n - f} \leq u_n$ for all $n \in \N$. Thus $f_n \to f$ in order as $n \to \infty$.
\end{proof}

Looking at the abstract setup from Section~\ref{section:an-abstract-banach-principle}, we need a locally solid topology on $E$. For that purpose, we endow $E$ with the topology of local convergence in measure and prove that this topology satisfies the abstract conditions mentioned in General Assumption~\ref{general-assumption:standing}. We note in passing that it can be even shown that it is the coarsest locally solid Hausdorff topology on $L^0(\Omega)$ (cf.\ \cite[Theorem~7.74]{Aliprantis2003}).

\begin{proposition} \label{proposition:measure-topology}
    Let $(\Omega, \Sigma, \mu)$ be a $\sigma$-finite measure space and let $E \coloneqq L^0(\Omega)$. Then the topology of local convergence in measure on $E$ is a metrisable, locally solid Hausdorff topology and has the $\sigma$-Lebesgue property.
\end{proposition}

\begin{proof}
    Since $(\Omega,\Sigma,\mu)$ is $\sigma$-finite, there exists an increasing sequence $(A_n)_{n \in \N}$ in $\Sigma$ such that
    \begin{align*}
        \Omega=\bigcup_{n=1}^\infty A_n \quad \text{and} \quad \mu(A_n) < \infty
    \end{align*}
    for all $n \in \N$. We first show that the topology of local convergence in measure is metrisable. For each $n\in\mathbb N$, consider
    \begin{align*}
        d_n(f,g) \coloneqq \int_{A_n} \frac{\abs{f - g}}{1 + \abs{f - g}} \, \ud \mu.
    \end{align*}
    It is clear that $d_n$ is a metric inducing convergence in measure on the finite measure space $(A_n, \Sigma|_{A_n}, \mu|_{A_n})$ (cf.\ the discussion before \cite[Theorem~7.74]{Aliprantis2003}). Now define
    \begin{align*}
        d(f,g) \coloneqq \sum_{n = 1}^{\infty} 2^{-n} \frac{d_n(f, g)}{1 + d_n(f, g)}.
    \end{align*}
    Then for $(f_n)_{n \in \N}$ in $E$ and $f \in E$, one has $d(f_k, f) \to 0$ if and only if $d_n(f_k, f) \to 0$ for all $n \in \N$. Hence $d$ induces the topology of local convergence in measure, so the topology is metrisable.

    To see that the topology is Hausdorff, note that if $f,g \in E$ satisfy $d(f, g)=0$, then $d_n(f, g) = 0$ for all $n \in \N$. Thus, $f = g$ almost everywhere on each of the $A_n$. As $\Omega=\bigcup_{n=1}^{\infty}A_n$, it follows that $f=g$ almost everywhere, i.e., the topology generated by $d$ is Hausdorff.

    Next we show that the topology of local convergence in measure is locally solid. A neighbourhood basis of $0$ is given by the sets
    \begin{align*}
        U(A,\varepsilon,\delta) \coloneqq \set{f \in E : \mu(A \cap \set{\abs f > \varepsilon })<\delta},
    \end{align*}
    where $A \in \Sigma$ has finite measure and $\varepsilon, \delta > 0$. Suppose that $g \in U(A,\varepsilon,\delta)$ and that $\abs f \leq \abs g$. Then
    \begin{align*}
        \set{\abs f > \varepsilon} \subseteq \set{\abs g > \varepsilon},
    \end{align*}
    and therefore
    \begin{align*}
        \mu(A\cap\set{\abs f > \varepsilon}) \leq \mu(A\cap\set{\abs g > \varepsilon}) < \delta.
    \end{align*}
    Hence $f \in U(A,\varepsilon,\delta)$. Thus every basic neighbourhood of $0$ is solid, so the topology of local convergence in measure is locally solid.

    Finally, it is shown that the topology has the $\sigma$-Lebesgue property. So pick a sequence $(f_n)_{n \in \N}$ in $E$ such that $f_n \downarrow 0$. Then there exists a null set $N \subseteq \Omega$ such that $f_n(x) \downarrow 0$ for all $x \in \Omega \setminus N$. Let $A \in \Sigma$ such that $\mu(A) < \infty$ and $\varepsilon > 0$. Since $\one_{A \cap \set{f_n > \varepsilon}} \to 0$ almost everywhere on $A$, dominated convergence yields
    \begin{align*}
        \mu (A \cap\set{f_n > \varepsilon}) = \int_A \one_{\set{f_n > \varepsilon}} \, \ud\mu \longrightarrow 0
    \end{align*}
    and thus $f_n \to 0$ locally in measure as $n \to \infty$. Hence, the topology of local convergence in measure has the $\sigma$-Lebesgue property.
\end{proof}

\subsection*{The classical setting revisited} \label{subsection:classical-banach-principle-and-hopf-dunford-schwartz}

We specialise the abstract results of the previous sections to recover the classical theory outlined, e.g., in \cite[Chapter~11]{Eisner2015}.

Let $(\Omega, \Sigma, \mu)$ be a $\sigma$-finite measure space and let $(T_n)_{n \in \N}$ be a sequence of bounded operators on $L^p(\Omega)$, $1 \leq p < \infty$. Following \cite[Definition~11.8]{Eisner2015}, the maximal operator $T^\star f \coloneqq \sup_{n \in \N} \abs{T_nf}$ is said to satisfy an \emph{abstract maximal inequality} if there is a function $c \colon (0,\infty) \to [0,\infty)$ with $c(\lambda) \to 0$ as $\lambda \to \infty$ such that
\begin{align*}
    \mu(\set{T^\star f > \lambda}) \leq c(\lambda)
\end{align*}
for all $\lambda > 0$ and $f \in L^p(\Omega)$ with $\norm f_p \leq 1$. In this setting, Banach's classical principle takes the following form (cf.\ \cite[Proposition~11.9]{Eisner2015}).

\begin{corollary}[Banach's principle, classical form] \label{corollary:classical-banach-principle}
    Let $(T_n)_{n \in \N}$ be a sequence of bounded operators on $L^p(\Omega)$, $1 \leq p < \infty$. Further suppose that $T^\star$ satisfies an abstract maximal inequality. Then
    \begin{align*}
        F \coloneqq \set{f \in L^p(\Omega) : (T_nf)_{n \in \N} \text{ converges almost everywhere}}
    \end{align*}
    is a closed subspace of $L^p(\Omega)$.
\end{corollary}

\begin{proof}
    Set $X \coloneqq L^p(\Omega)$ and $E \coloneqq L^0(\Omega)$. By Proposition~\ref{proposition:l0-lattice-properties}, $E$ is Dedekind $\sigma$-complete, and by Proposition~\ref{proposition:measure-topology} the topology $\tau$ of local convergence in measure on $E$ is metrisable, Hausdorff, locally solid and has the $\sigma$-Lebesgue property. Since, by Chebyshev's inequality (see \cite[6.17]{Folland1999}), $L^p$-convergence implies local convergence in measure, each $T_n \colon X \to E$ is continuous with respect to the norm topology on $X$ and $\tau$ on $E$.

    By homogeneity of $T^\star$, the maximal inequality extends from the unit ball to all of $X$ and, letting $\lambda \to \infty$, one obtains $T^\star f < \infty$ almost everywhere. By Proposition~\ref{proposition:abstract-banach-principle}(i), $(T_nf)_{n \in \N}$ is thus order bounded in $E$ for every $f \in X$, so General Assumption~\ref{general-assumption:standing} holds. Thus, Theorem~\ref{theorem:banach-principle}(i) shows that
    \begin{align*}
        C \coloneqq \set{f \in X : (T_nf)_{n\in\N} \text{ is order convergent in } E}
    \end{align*}
    is a closed subspace of $X$. Since every orbit $(T_n f)_{n \in \N}$, $f \in X$, is order bounded in $E$, Proposition~\ref{proposition:abstract-banach-principle}(iii) shows that order convergence and almost everywhere convergence coincide on all of $X$, whence $C = F$.
\end{proof}

\begin{remark}
    Corollary~\ref{corollary:classical-banach-principle} recovers \cite[Proposition~11.9]{Eisner2015}. The classical statement postulates the quantitative bound $c(\lambda)$ outright, whereas Theorem~\ref{theorem:banach-principle} only requires the pointwise order boundedness of each orbit, i.e., General Assumption~\ref{general-assumption:standing}(d). By Lemma~\ref{lemma:maximal-operator-continuity}, this pointwise hypothesis already forces $T^\star$ to be continuous at the origin via Baire's category theorem, with no quantitative rate needed; that an abstract maximal inequality is necessary and not merely sufficient for pointwise convergence results of this type is classical, going back to Stein \cite{Stein1961} and treated systematically in \cite[Chapter~1.7]{Krengel1985}. The function $c(\lambda)$ of \cite[Definition~11.8]{Eisner2015} is thus a convenient sufficient condition for this continuity rather than an intrinsic requirement of the principle.
\end{remark}

We now turn to the Hopf--Dunford--Schwartz theorem (see \cite[Theorem~11.4]{Eisner2015}). Recall that a positive operator $T$ on $L^1(\Omega)$ is called a \emph{Dunford--Schwartz operator} if $\norm{Tf}_1 \leq \norm f_1$ for all $f \in L^1(\Omega)$ and $\norm{Tf}_\infty \leq \norm f_\infty$ for all $f \in L^1(\Omega) \cap L^\infty(\Omega)$.

\begin{corollary}[Hopf--Dunford--Schwartz] \label{corollary:hopf-dunford-schwartz}
    Let $(\Omega, \Sigma, \mu)$ be a finite measure space and let $T$ be a Dunford--Schwartz operator on $E \coloneqq L^1(\Omega)$. Then the Cesàro means $A_n f = \frac 1n \sum_{k=0}^{n-1}T^kf$ converge almost everywhere for every $f \in E$.
\end{corollary}

\begin{proof}
    Since $\mu$ is finite, one has $\one \in E \cap L^\infty(\Omega)$, and positivity of $T$ together with the $L^\infty$-contractivity yields $T\one \leq \norm{\one}_\infty \one = \one$; thus $e \coloneqq \one$ is a super fixed point of $T$. As an $L^1$-contraction, $T$ is power-bounded on $E$, and $T$ is mean ergodic on $E$ with mean ergodic projection $P$ by \cite[Theorem~8.24]{Eisner2015}.

    Let $\varphi \colon E \to \R, \, f \mapsto \int_\Omega f \, \ud \mu$. Then $\varphi$ is strictly positive and $\sigma$-order continuous, and for $0 \leq f \in E$ one has
    \begin{align*}
        \dual{f, T'\varphi} = \dual{Tf,\varphi} = \int_\Omega Tf \, \ud\mu = \norm{Tf}_1 \leq \norm f_1 = \int_\Omega f \, \ud\mu = \dual{f,\varphi},
    \end{align*}
    so $\varphi$ is a super fixed point of $T'$. Thus $(e,\varphi)$ is a superinvariant pair for $T$ (cf.\ Remark~\ref{remark:calderon}). Since $E = L^1(\Omega)$ has order continuous norm and weak unit $e$, it satisfies the hypotheses of Theorem~\ref{theorem:individual-ergodic}, so by Corollary~\ref{corollary:individual-from-maximal} the sequence $(A_n f)_{n\in\N}$ is order convergent to $Pf$ in $E^u$ and, in particular, $A_n f \uoto Pf$ in $E$ for every $f \in E$.

    Moreover, $E^u = L^0(\Omega)$ by Proposition~\ref{proposition:l0-universal-completion}. Order convergence in $E^u$ implies uo-convergence there, and Proposition~\ref{proposition:abstract-banach-principle}(ii) identifies uo-convergence in $L^0(\Omega)$ with almost everywhere convergence. Hence $A_nf \to Pf$ almost everywhere for every $f \in E$.
\end{proof}

\begin{remark}
    Corollary~\ref{corollary:hopf-dunford-schwartz} recovers \cite[Theorem~11.4]{Eisner2015} in the case of a finite measure space, which is also the case treated first in \cite[Chapter~11]{Eisner2015} before an exhaustion argument extends it to general measure spaces; we do not pursue this extension here, since it requires working with $L^0$-spaces over measure spaces that need not be $\sigma$-finite, outside the scope of Proposition~\ref{proposition:l0-universal-completion}. Together with Remark~\ref{remark:maximal-ergodic-classical}, the same superinvariant pair $(e,\varphi) = (\one, \int_\Omega (\cdot)\,\ud\mu)$ also recovers the maximal ergodic theorem and the weak type $(1,1)$-maximal inequality
    \begin{align*}
        \lambda \mu(\set{M\abs f > \lambda}) \leq \norm{f}_1 \qquad (\lambda > 0)
    \end{align*}
    for all $f \in L^1(\Omega)$ of \cite[Theorem~11.11 and Corollary~11.12]{Eisner2015}.
\end{remark}

\begin{remark} \label{remark:akcoglu}
    The superinvariant pair framework of Section~\ref{section:an-abstract-maximal-ergodic-theorem} is intrinsically tied to the Dunford--Schwartz setting: by Remark~\ref{remark:calderon}, a superinvariant pair encodes simultaneous contractivity with respect to an $L^1$- and an $L^\infty$-norm. It therefore does not cover Akcoglu's pointwise ergodic theorem (see \cite{Akcoglu1975}), which asserts almost everywhere convergence of the Cesàro means $A_nf$ for every positive linear contraction $T$ on $L^p(\Omega)$, $1 < p < \infty$, without assuming any simultaneous $L^1$- or $L^\infty$-contractivity. A general such $T$ admits no canonical superinvariant pair on $L^p(\Omega)$ itself, and Akcoglu's proof instead proceeds via a dilation of $T$ to an invertible isometry on an auxiliary measure space (cf.\ \cite{Akcoglu1977}), to which the classical Dunford--Schwartz theory is then applied and transferred back. This dilation is an independent construction and not covered by the vector lattice methods developed in this article.
\end{remark}

We close this section with a proof of Doob's martingale theorem (see \cite[11.5]{Williams1991}).

\begin{proposition}[Doob] \label{proposition:doob-martingale}
    Let $(\Omega, \Sigma, \mu)$ be a probability space and let $(\calF_n)_{n \in \N}$ be a filtration in $\Sigma$ such that $\Sigma = \sigma(\bigcup_{n \in \N} \calF_n)$. Set $T_n f \coloneqq \bbE(f \mid \calF_n)$ for $f \in L^1(\Omega)$. Then $T_n f \to f$ almost everywhere for each $f \in L^1(\Omega)$.
\end{proposition}

\begin{proof}
    Each $T_n$ is a positive contraction on $L^1(\Omega)$ by monotonicity and Jensen's inequality for conditional expectations, and by Doob's maximal inequality (see \cite[Theorem~19.12]{Schilling2005}) one has
    \begin{align*}
        \mu(\set{\sup_{n \in \N} \abs{T_nf} > \lambda}) \leq \frac{1}{\lambda} \norm{f}_1, \qquad (\lambda > 0),
    \end{align*}
    for all $f \in L^1(\Omega)$. This is an abstract maximal inequality. By Corollary~\ref{corollary:classical-banach-principle}, the set
    \begin{align*}
        F \coloneqq \set{f \in L^1(\Omega) : (T_nf)_{n \in \N} \text{ converges almost everywhere}}
    \end{align*}
    is a closed subspace of $L^1(\Omega)$.

    Set $D \coloneqq \bigcup_{m \in \N} L^1(\Omega, \calF_m, \mu|_{\calF_m})$. If $f \in L^1(\Omega, \calF_m, \mu|_{\calF_m})$, the tower property of conditional expectation yields $T_nf = f$ for all $n \geq m$, so $(T_nf)_{n \in \N}$ is eventually constant and $f \in F$. Thus, $D \subseteq F$. Since $\bigcup_m \calF_m$ generates $\Sigma$, $D$ is norm dense in $L^1(\Omega)$ by a standard result from measure theory. As $F$ is a closed subspace containing the dense set $D$, it follows that $F = L^1(\Omega)$.
\end{proof}

\begin{remark}
    We emphasise that the application of the abstract Banach principle in the proof of Proposition~\ref{proposition:doob-martingale} is genuine but somewhat tautological: the hard analytic input, Doob's maximal inequality, is precisely the kind of maximal estimate that the abstract theory of Section~\ref{section:an-abstract-banach-principle} is designed to package. In particular, the maximal inequality used here is external input and is not a consequence of Theorem~\ref{theorem:maximal-ergodic}; the latter concerns Cesàro means of a single power-bounded operator, whereas the conditional expectations $(T_n)_{n\in\N}$ do not arise as iterates of a fixed operator. The role of the abstract principle here is confined to upgrading the maximal inequality plus convergence on the dense class $D$ to convergence on all of $L^1(\Omega)$.
\end{remark}

\bibliographystyle{plain}
\bibliography{literature}

@book{Aliprantis2003,
	author = {Aliprantis, Charalambos D. and Burkinshaw, Owen},
	title = {Locally Solid {R}iesz Spaces with Applications to Economics},
	series = {Mathematical Surveys and Monographs},
	volume = {105},
	edition = {2},
	publisher = {American Mathematical Society},
	address = {Providence, RI},
	year = {2003},
}

@book{Eisner2015,
	author = {Eisner, Tanja and Farkas, B\'alint and Haase, Markus and Nagel, Rainer},
	title = {Operator Theoretic Aspects of Ergodic Theory},
	series = {Graduate Texts in Mathematics},
	volume = {272},
	publisher = {Springer},
	address = {Cham},
	year = {2015},
}

@article{Gao2017,
	author = {Gao, Niushan and Troitsky, Vladimir G. and Xanthos, Foivos},
	title = {Uo-convergence and its applications to {C}es\`aro means in {B}anach lattices},
	journal = {Israel Journal of Mathematics},
	volume = {220},
	number = {2},
	pages = {649--689},
	year = {2017},
}

@book{Krengel1985,
	author = {Krengel, Ulrich},
	title = {Ergodic Theorems},
	series = {De Gruyter Studies in Mathematics},
	volume = {6},
	publisher = {Walter de Gruyter},
	address = {Berlin},
	year = {1985},
}

@book{MeyerNieberg1991,
	author = {Meyer-Nieberg, Peter},
	title = {Banach Lattices},
	series = {Universitext},
	publisher = {Springer},
	address = {Berlin},
	year = {1991},
}

@article{Stein1961,
	author = {Stein, Elias M.},
	title = {On limits of sequences of operators},
	journal = {Annals of Mathematics (2)},
	volume = {74},
	pages = {140--170},
	year = {1961},
}

@book{Doob1953,
  author    = {Doob, Joseph L.},
  title     = {Stochastic Processes},
  publisher = {Wiley},
  year      = {1953}
}

@article{Stein1976,
  author  = {Stein, Elias M.},
  title   = {Maximal functions: spherical means},
  journal = {Proc. Natl. Acad. Sci. USA},
  volume  = {73},
  pages   = {2174--2175},
  year    = {1976}
}

@article{Bourgain1986,
  author  = {Bourgain, Jean},
  title   = {Averages in the plane over convex curves and maximal operators},
  journal = {J. Analyse Math.},
  volume  = {47},
  pages   = {69--85},
  year    = {1986}
}

@article{Billard1967,
  author  = {Billard, Pierre},
  title   = {Sur la convergence presque partout des s\'eries de {F}ourier-{W}alsh
             des fonctions de l'espace {$L^2(0,1)$}},
  journal = {Studia Math.},
  volume  = {28},
  pages   = {363--388},
  year    = {1967}
}

@article{Goldstein2000,
  title={Banach principle in the space of $\tau$-measurable operators},
  author={Goldstein, Michael and Litvinov, Semyon},
  journal={Studia Mathematica},
  volume={143},
  number={1},
  pages={33--41},
  year={2000}
}

@book{Vulikh1967,
  title={Introduction to the theory of partially ordered spaces},
  author={Vulikh, Boris Z. and Boron, Leo F. and Zaanen, Adriaan C.},
  publisher={Noordhoff, Groningen},
  year={1967}
}

@book{Aliprantis2006,
  title={Positive operators},
  author={Aliprantis, Charalambos D and Burkinshaw, Owen},
  volume={119},
  year={2006},
  publisher={Springer Science \& Business Media}
}

@article{Cerreia2022,
  title={A Characterization of the Vector Lattice of Measurable Functions},
  author={Cerreia-Vioglio, Simone and Leonetti, Paolo and Maccheroni, Fabio},
  journal={Milan Journal of Mathematics},
  volume={90},
  number={1},
  pages={291--301},
  year={2022},
  publisher={Springer}
}

@book{Zaanen2012,
  title={Introduction to operator theory in Riesz spaces},
  author={Zaanen, Adriaan C},
  year={2012},
  publisher={Springer Science \& Business Media}
}

@book{Folland1999,
  title={Real analysis: modern techniques and their applications},
  author={Folland, Gerald B.},
  year={1999},
  publisher={John Wiley \& Sons}
}

@book{Schilling2005,
  title={Measures, integrals and martingales},
  author={Schilling, Ren{\'e} L.},
  year={2005},
  publisher={Cambridge University Press}
}

@book{Williams1991,
  title={Probability with martingales},
  author={Williams, David},
  year={1991},
  publisher={Cambridge university press}
}

@Book{Batkai2017,
  author     = {B\'{a}tkai, Andr\'{a}s and Kramar Fijav\v{z}, Marjeta and Rhandi, Abdelaziz},
  publisher  = {Birkh\"{a}user/Springer, Cham},
  title      = {Positive operator semigroups},
  year       = {2017},
  isbn       = {978-3-319-42811-6},
  note       = {From finite to infinite dimensions, With a foreword by Rainer Nagel and Ulf Schlotterbeck},
  series     = {Operator Theory: Advances and Applications},
  volume     = {257},
  doi        = {10.1007/978-3-319-42813-0},
  mrclass    = {47-02 (15-02 47B65 47D06)},
  mrreviewer = {Christoph Kriegler},
  pages      = {xvii+364},
  url        = {https://doi.org/10.1007/978-3-319-42813-0},
}

@Book{Schaefer1974,
  author     = {Schaefer, Helmut H.},
  publisher  = {Springer-Verlag, New York-Heidelberg},
  title      = {Banach lattices and positive operators},
  year       = {1974},
  series     = {Die Grundlehren der mathematischen Wissenschaften, Band 215},
  mrclass    = {46A40 (47B55 47D20)},
  mrreviewer = {A. C. Zaanen},
  pages      = {xi+376},
}

@Book{Zaanen1983,
  author     = {Zaanen, A. C.},
  publisher  = {North-Holland Publishing Co., Amsterdam},
  title      = {Riesz spaces. {II}},
  year       = {1983},
  isbn       = {0-444-86626-4},
  series     = {North-Holland Mathematical Library},
  volume     = {30},
  mrclass    = {46-02 (46A40 47-02 47B55)},
  mrnumber   = {704021},
  mrreviewer = {A. C. Thompson},
  pages      = {xi+720},
}

@Book{Luxemburg1971,
  author     = {Luxemburg, W. A. J. and Zaanen, A. C.},
  publisher  = {North-Holland Publishing Co., Amsterdam-London; American Elsevier Publishing Co., Inc., New York},
  title      = {Riesz spaces. {V}ol. {I}},
  year       = {1971},
  series     = {North-Holland Mathematical Library},
  mrclass    = {46A40 (06A65)},
  mrnumber   = {511676},
  mrreviewer = {S. J. Bernau},
  pages      = {xi+514},
}

@Book{Aliprantis1999,
  author    = {Aliprantis, Charalambos D. and Border, Kim C.},
  publisher = {Springer-Verlag, Berlin},
  title     = {{I}nfinite-{D}imensional {A}nalysis},
  year      = {1999},
  edition   = {Second},
  isbn      = {3-540-65854-8},
  note      = {A hitchhiker's guide},
  doi       = {10.1007/978-3-662-03961-8},
  mrclass   = {46-01 (00A05 28-01 46N10 47-01 54-01)},
  pages     = {xx+672},
  url       = {https://doi.org/10.1007/978-3-662-03961-8},
}

@article{Akcoglu1975,
  author  = {Akcoglu, Mustafa A.},
  title   = {A pointwise ergodic theorem in {$L_p$}-spaces},
  journal = {Canadian Journal of Mathematics},
  volume  = {27},
  number  = {5},
  pages   = {1075--1082},
  year    = {1975}
}

@article{Akcoglu1977,
  author  = {Akcoglu, Mustafa A. and Sucheston, Louis},
  title   = {Dilations of positive contractions on {$L_p$} spaces},
  journal = {Canadian Mathematical Bulletin},
  volume  = {20},
  number  = {3},
  pages   = {285--292},
  year    = {1977}
}

@book{Eisner2025,
  title={A journey through ergodic theorems},
  author={Eisner, Tanja and Farkas, B{\'a}lint},
  year={2025},
  publisher={Springer}
}

@article{Dunford1956,
  author  = {Dunford, Nelson and Schwartz, Jacob T.},
  title   = {Convergence almost everywhere of operator averages},
  journal = {J. Rational Mech. Anal.},
  volume  = {5},
  pages   = {129--178},
  year    = {1956}
}

@article{Hopf1954,
  author  = {Hopf, Eberhard},
  title   = {The general temporally discrete {M}arkoff process},
  journal = {J. Rational Mech. Anal.},
  volume  = {3},
  pages   = {13--45},
  year    = {1954}
}

@article{Garsia1965,
  author  = {Garsia, Adriano M.},
  title   = {A simple proof of {E}. {H}opf's maximal ergodic theorem},
  journal = {J. Math. Mech.},
  volume  = {14},
  pages   = {381--382},
  year    = {1965}
}

@article{Yosida1939,
  author  = {Yosida, K\={o}saku and Kakutani, Shizuo},
  title   = {Birkhoff's ergodic theorem and the maximal ergodic theorem},
  journal = {Proc. Imp. Acad.},
  volume  = {15},
  number  = {6},
  pages   = {165--168},
  year    = {1939}
}

@article{Nakano1948,
  author  = {Nakano, Hidegor{\^o}},
  title   = {Ergodic theorems in semi-ordered linear spaces},
  journal = {Ann. of Math.},
  volume  = {49},
  number  = {2},
  pages   = {538--556},
  year    = {1948}
}

@article{DeMarr1964,
  author  = {DeMarr, Ralph},
  title   = {Partially ordered linear spaces and locally convex linear topological spaces},
  journal = {Illinois J. Math.},
  volume  = {8},
  pages   = {601--606},
  year    = {1964}
}

@article{Wickstead1977,
  author  = {Wickstead, Anthony W.},
  title   = {Weak and unbounded order convergence in {B}anach lattices},
  journal = {J. Austral. Math. Soc. Ser. A},
  volume  = {24},
  number  = {3},
  pages   = {312--319},
  year    = {1977}
}

@article{GaoXanthos2014,
  author  = {Gao, Niushan and Xanthos, Foivos},
  title   = {Unbounded order convergence and application to martingales without probability},
  journal = {J. Math. Anal. Appl.},
  volume  = {415},
  number  = {2},
  pages   = {931--947},
  year    = {2014}
}

@article{Gao2014,
  author  = {Gao, Niushan},
  title   = {Unbounded order convergence in dual spaces},
  journal = {J. Math. Anal. Appl.},
  volume  = {419},
  number  = {1},
  pages   = {347--354},
  year    = {2014}
}

@article{Kandic2017,
  author  = {Kandi{\'c}, Marko and Marabeh, Mohammad A. A. and Troitsky, Vladimir G.},
  title   = {Unbounded norm topology in {B}anach lattices},
  journal = {J. Math. Anal. Appl.},
  volume  = {451},
  number  = {1},
  pages   = {259--279},
  year    = {2017}
}

@article{DengOBrienTroitsky2017,
  author  = {Deng, Yang and O'Brien, Michael and Troitsky, Vladimir G.},
  title   = {Unbounded norm convergence in {B}anach lattices},
  journal = {Positivity},
  volume  = {21},
  number  = {3},
  pages   = {963--974},
  year    = {2017}
}

@article{GaoLeungXanthos2018,
  author  = {Gao, Niushan and Leung, Denny H. and Xanthos, Foivos},
  title   = {Duality for unbounded order convergence and applications},
  journal = {Positivity},
  volume  = {22},
  number  = {3},
  pages   = {711--725},
  year    = {2018}
}

@article{Li2018,
  author  = {Li, Hui and Chen, Zili},
  title   = {Some loose ends on unbounded order convergence},
  journal = {Positivity},
  volume  = {22},
  number  = {1},
  pages   = {83--90},
  year    = {2018}
}

@article{Kandic2018,
  author  = {Kandi{\'c}, Marko and Taylor, Mitchell A.},
  title   = {Metrizability of minimal and unbounded topologies},
  journal = {J. Math. Anal. Appl.},
  volume  = {466},
  number  = {1},
  pages   = {144--159},
  year    = {2018}
}

@article{Birkhoff1931,
  author  = {Birkhoff, George D.},
  title   = {Proof of the ergodic theorem},
  journal = {Proc. Natl. Acad. Sci. USA},
  volume  = {17},
  pages   = {656--660},
  year    = {1931}
}

@article{Banach1926,
  author  = {Banach, Stefan},
  title   = {Sur la convergence presque partout de fonctionnelles lin{\'e}aires},
  journal = {Bull. Sci. Math.},
  volume  = {50},
  series  = {2},
  pages   = {27--32, 36--43},
  year    = {1926}
}

@article{VanDerWalt2018,
  title={The universal completion of ${C}({X})$ and unbounded order convergence},
  author={van der Walt, Jan Harm},
  journal={Journal of Mathematical Analysis and Applications},
  volume={460},
  number={1},
  pages={76--97},
  year={2018},
  publisher={Elsevier}
}

@article{Taylor2018,
  title={Completeness of unbounded convergences},
  author={Taylor, Mitchell A.},
  journal={Proceedings of the American Mathematical Society},
  volume={146},
  number={8},
  pages={3413--3423},
  year={2018}
}

@article{Azouzi2019,
  title = {Completeness for vector lattices},
  journal = {Journal of Mathematical Analysis and Applications},
  volume = {472},
  number = {1},
  pages = {216-230},
  year = {2019},
  issn = {0022-247X},
  author = {Youssef Azouzi}
}

@article{Taylor2019,
  title={Unbounded topologies and uo-convergence in locally solid vector lattices},
  author={Taylor, Mitchell A.},
  journal={Journal of Mathematical Analysis and Applications},
  volume={472},
  number={1},
  pages={981--1000},
  year={2019},
  publisher={Elsevier}
}

@article{Kuo2007,
  author  = {Kuo, Wen-Chi and Labuschagne, Coenraad C. A. and Watson, Bruce A.},
  title   = {Ergodic theory and the strong law of large numbers on {R}iesz spaces},
  journal = {J. Math. Anal. Appl.},
  volume  = {325},
  number  = {1},
  pages   = {422--437},
  year    = {2007},
  doi     = {10.1016/j.jmaa.2006.01.056}
}

@article{Kuo2004,
  author  = {Kuo, Wen-Chi and Labuschagne, Coenraad C. A. and Watson, Bruce A.},
  title   = {Discrete-time stochastic processes on {R}iesz spaces},
  journal = {Indag. Math. (N.S.)},
  volume  = {15},
  number  = {3},
  pages   = {435--451},
  year    = {2004}
}

@article{Azouzi2023,
  author  = {Azouzi, Youssef and Ben Amor, Mohamed A. and Homann, Jonathan and Masmoudi, Marwa and Watson Bruce A.},
  title   = {The {K}ac formula and {P}oincar\'e recurrence theorem in {R}iesz spaces},
  journal = {Proc. Amer. Math. Soc. Ser. B},
  volume  = {10},
  pages   = {182--194},
  year    = {2023}
}

@article{Azouzi2024,
  author  = {Azouzi, Youssef and Masmoudi, Marwa},
  title   = {Some characterizations of ergodicity in {R}iesz spaces},
  journal = {Quaestiones Mathematicae},
  year    = {2024}
}

@article{Azouzi2025,
  author  = {Azouzi, Youssef and Masmoudi, Marwa and Watson, Bruce A.},
  title   = {A {K}akutani--{R}okhlin decomposition for conditionally ergodic process in the measure-free setting of vector lattices},
  journal = {Ergodic Theory Dynam. Systems},
  volume  = {45},
  pages   = {3600--3618},
  year    = {2025}
}

@article{Hardy1930,
  author  = {Hardy, Godfrey H. and Littlewood, John E.},
  title   = {A maximal theorem with function-theoretic applications},
  journal = {Acta Math.},
  volume  = {54},
  number  = {1},
  pages   = {81--116},
  year    = {1930},
  doi     = {10.1007/BF02547518}
}

@book{Stein1970a,
  title={Singular integrals and differentiability properties of functions},
  author={Stein, Elias M.},
  volume={2},
  year={1970},
  publisher={Princeton university press}
}

@article{Carbery1988,
  author  = {Carbery, Anthony and Rubio de Francia, Jos\'e L. and Vega, Luis},
  title   = {Almost everywhere summability of {F}ourier integrals},
  journal = {J. London Math. Soc. (2)},
  volume  = {38},
  number  = {3},
  pages   = {513--524},
  year    = {1988}
}

\end{document}